\documentclass[a4paper,12pt]{article}
\usepackage[english]{babel}
\usepackage{amsmath,amssymb,amsthm,xcolor,latexsym,dsfont}
\usepackage[T1]{fontenc}
\usepackage{authblk}
\usepackage{comment}
\usepackage[margin=2.5cm]{geometry}
\usepackage[hidelinks]{hyperref}
\usepackage{svg}
\usepackage{caption}

\allowdisplaybreaks

\newcommand{\assign}{:=}
\newcommand{\mathd}{\mathrm{d}}

\newtheorem{theorem}{Theorem}[section]
\newtheorem*{theorem*}{Theorem}
\newtheorem{corollary}[theorem]{Corollary}
\newtheorem{lemma}[theorem]{Lemma}
\newtheorem{proposition}[theorem]{Proposition}
\newtheorem{remark}[theorem]{Remark}
\newtheorem{definition}[theorem]{Definition}
\newtheorem{example}[theorem]{Example}
\newtheorem{assumption}{Assumption}

\newcommand{\keywords}[1]{%
  \par\smallskip
  \noindent\textbf{Keywords: }#1
}

\newcommand{\MSC}[2][]{%
  \par\smallskip
  \noindent\textbf{MSC Classification: }#2
}

\newcommand{\R}{\mathbb{R}}
\newcommand{\T}{\mathbb{T}}
\newcommand{\C}{\mathbb{C}}
\newcommand{\N}{\mathbb{N}}
\newcommand{\Z}{\mathbb{Z}}
\newcommand{\E}{\mathbb{E}}

\title{On the role of positivity preservation for high order approximations of the Dean--Kawasaki equation}

\date{}

\author[1]{Ana Damnjanovi\'c}
\author[2]{Ana Djurdjevac}
\author[1,3]{Nicolas Perkowski}

\affil[1]{Freie Universit\"at Berlin, Arnimallee 7, 14195 Berlin, Germany}
\affil[2]{University of Oxford, Mathematical Institute, Oxford, UK }
\affil[3]{Max Planck Institute for Mathematics in the Sciences, Leipzig, Germany}

\begin{document}
\maketitle
\begin{abstract}

We study a spectral regularization of the Dean--Kawasaki equation and quantify how the failure of positivity preservation affects its weak approximation of the empirical measure of independent Brownian particles.  For initial densities bounded away from zero, we prove a uniform-in-time weak error 
measured through the Laplace transform and prove a superpolynomial convergence rate for smooth test functions. When the initial density is allowed to vanish, the negative part of the regularized solution leads to weaker upper bounds, and a one-dimensional example gives a lower error bound ruling out superpolynomial convergence. 
We also present numerical experiments that confirm the theoretical results and illustrate  main  observations.


\MSC[]{60H15, 60J60, 60H35, 65C30, 60K35}
\end{abstract}

\section*{Introduction}

The Dean--Kawasaki equation is one of the central equations of fluctuating
hydrodynamics. In its simplest form it describes the evolution of the empirical
distribution of $N$ independent Brownian particles and is formally given by
\begin{equation}\label{eq:DK-intro}
    \partial_t \rho
    =
    \frac12 \Delta \rho
    +
    \frac1{\sqrt N}\nabla\cdot\bigl(\sqrt{\rho}\,\xi\bigr),
\end{equation}
on $\R_+ \times \T^d$, where $\xi$ denotes vector-valued space-time white noise. 
Dean~\cite{Dean1996} derived the equation with the help of It\^o's formula, showing that the empirical distribution is a martingale solution to the  stochastic conservation law~\eqref{eq:DK-intro}. The equation is therefore exact
at the microscopic level, but only as an equation for an atomic random measure.
Remarkably, Kawasaki~\cite{Kawasaki1994} had previously derived the same equation from a different, mesoscopic perspective. He coarse-grained the empirical measure by averaging it over small spatial scales and formally applied a local equilibrium closure to derive a closed infinite-dimensional Fokker--Planck equation for the evolution of the probability distribution of the regularized density. The fluctuations in Kawasaki's equation are described by the same conservative square-root noise as in Dean's equation. See the recent survey by Illien~\cite{Illien2025} for further background on the Dean--Kawasaki equation and its applications in theoretical physics.

Unfortunately, the Dean--Kawasaki equation is too singular to be interpreted as a
density-valued SPDE. The noise coefficient $\sqrt{\rho}$ formally is a non-linear function of a measure and not even locally Lipschitz continuous, and the equation is scaling
supercritical in the sense of singular SPDEs~\cite{Hairer2014, Gubinelli2015Paracontrolled}. But as Dean implicitly observed, and Konarovskyi, Lehmann and von Renesse~\cite{Konarovskyi2019} explicitly pointed out, the equation has a meaningful interpretation in the martingale sense. Konarovskyi et al. showed a duality result which is analogous to results for superprocesses and proves uniqueness in law for martingale solutions. However, the striking triviality result of Konarovskyi et al. shows that the only solutions to this martingale problem are exactly the empirical distributions of the particle system, and in contrast to stochastic transport equations~\cite{Kunita1990} it is impossible to consider density-valued initial conditions. In that sense Dean's equation is an exact reformulation of the particle dynamics, and it is unclear how to rigorously interpret  Kawasaki's mesoscopic equation without further regularization.

The Dean--Kawasaki equation is therefore not a coarse-grained model. To coarse-grain, we have to regularize the equation. Recent mathematical work has considered ad-hoc truncations of the Dean--Kawasaki equation and compared their behavior for large $N$ with that of the particle system with many particles. The approximation quality was verified by comparing large deviation rate functions for $N\to \infty$ \cite{Fehrman2023, Fehrman2024}, by verifying consistent fluctuation dissipation relations \cite{Donev2010, Bell2026}, or, taking a numerical analysis perspective, by deriving weak error bounds for the expectations of suitable smooth observables \cite{Cornalba2023, Cornalba2023Density, Djurdjevac2024, Djurdjevac2026}.

To obtain a meaningful interpretation of the truncated Dean--Kawasaki equation as an approximation of the particle system, it is desirable that solutions are positive and that the mass is conserved. Indeed, $\rho$ represents  a density approximation to the empirical distribution, so $\rho(t,x) \geq 0$ and $\int \rho(t,x)\mathd x = 1$.

The available approaches to regularizing the Dean--Kawasaki equation can be grouped according to whether they preserve this positivity. The finite-difference approximations developed in  \cite{Cornalba2023, Cornalba2023Density} achieve high-order weak approximation in high-density regimes. However, because the finite difference divergence operator is nonlocal, they do not preserve positivity and even for nonnegative initial data the approximations can take negative values. The analysis in~\cite{Cornalba2023, Cornalba2023Density} therefore assumes that the initial density is bounded away from zero and in that case proves exponential tail bounds for the negative part. 

A different approach was pursued in \cite{Djurdjevac2024, Djurdjevac2026}, where the noise was truncated and the square-root singularity was regularized. This regularization preserves the positivity and the mass, and it allows initial densities that locally vanish. But the existence of solutions requires a strong truncation of the noise because otherwise the noise term might inject more energy than the Laplacian dissipates which would break the parabolic nature of the equation. And this strong truncation of the noise limits the approximation quality, which is uniform across all initial distributions but for initial densities bounded away from zero achieves a worse rate than \cite{Cornalba2023, Cornalba2023Density}.

Another approach is taken in \cite{Fehrman2024, Fehrman2023}, where the It\^o noise in~\eqref{eq:DK-intro} is replaced by Stratonovich noise. This allows them to keep the square root coefficient and to truncate the noise less severely than in~\cite{Djurdjevac2024, Djurdjevac2026} while retaining the parabolic nature of the equation and preserving the positivity and mass of the solution. In~\cite{Fehrman2023, Wu2022, Ji2026} it is shown that this Stratonovich equation satisfies the same large deviation principle as the particle system in the limit $N \to \infty$. However, for finite $N$ the weak error for expectations of smooth nonlinear test functions is expected to be of the order of magnitude of the It\^o-Stratonovich corrector, that is $O(N^{-1})$, which is the same order as for the heat equation, the deterministic limit for $N \to \infty$.

Thus, some basic questions remain open: What is the effect of violating positivity preservation if we consider initial conditions that locally have low density? Even though here we focus on non-interacting Brownian particles, this question is particularly relevant in the interacting case for which the interaction can create low density regions~\cite{Wehlitz2025}. Moreover, in the high density regime~\cite{Cornalba2023} construct for any convergence rate $p$ an approximation which achieves rate $p$, so another natural question is if we can find an approximation which achieves superpolynomial convergence rates in this setting. And, finally, whether there exist positivity-preserving approximations achieving a higher order of convergence than those in \cite{Djurdjevac2024}.

In this paper we address only the first two of these questions for a spectral regularization of the Dean--Kawasaki equation. We consider
\begin{equation}\label{eq:approx_DK}
    \mathd u_t = \frac{1}{2} \Delta u _t \mathd t + \frac{1}{\sqrt{N}} \nabla \cdot (K_\varepsilon \ast (f(u_t^+)\mathd W_t)) 
\end{equation}
where $K_\varepsilon$ is a suitable mollifier at spatial scale $\varepsilon$, with compactly supported Fourier transform that is constant near $0$, and $f$ is a continuous function with $f(0)=0$ chosen so as to mimic the Dean--Kawasaki noise coefficient. The canonical example is  $f(x) = \sqrt{x}$, although we may replace this by a Lipschitz approximation in order to obtain a unique solution. The positive part $(\cdot)^+$ is taken to ensure that the noise is switched off at negative values of the solution, to prevent too large negative values. The main difference with respect to the approximation of \cite{Djurdjevac2024} is that the Wiener process $W$ is not truncated, and the spectral regularization $K_\varepsilon \ast \cdot$ acts on the entire noise coefficient $f(u_t^+)\mathd W_t$. This regularization is closely related to a spectral Galerkin approximation, which we use in the numerical experiments.

This choice has the advantage that when testing $u$ against a test function $\varphi$, the regularization $K_\varepsilon$ is moved to $\varphi$ and we can exploit the regularity of the test function. On the other hand, the nonlocal nature of the convolution breaks the positivity preservation of~\cite{Djurdjevac2024}, and it is not true that a.s. $u_t \geq 0$ for $t>0$. However, the spectral regularization plays an important stabilizing role. The convolution with $K_\varepsilon$ truncates frequencies larger than $O(\varepsilon^{-1})$, which removes the singular small-scale fluctuations of the Dean--Kawasaki noise. 
In high-density regimes, where the typical density is large and the relative fluctuations are small, this regularization can make negative excursions rare or small in magnitude. Thus, while the regularized equation does not preserve positivity, the convolution by $K_\varepsilon$, together with the scaling $N^{-1/2}$, can guarantee positivity with high probability in regimes where the density is sufficiently high compared with the strength of the noise.

Returning to the two questions raised above, we address them as follows: We derive both upper and lower bounds on the approximation error in the low density regime, i.e. for initial densities that may locally vanish. The lower bound is derived for a concrete one-dimensional example. This constitutes the first analysis of non-positivity preserving approximations of the Dean--Kawasaki equation allowing low densities. Notably, in dimensions one and two the upper bound is worse than the one obtained by the positivity preserving approximation of \cite{Djurdjevac2024}. The lower bound does not quite match it, but it indicates that in low density regimes approximations that do not preserve positivity have limited approximation quality. On the other hand, in the high density regime $\rho_0 \ge \rho_{\min} > 0$ we show that on smooth observables, the approximation~\eqref{eq:approx_DK} achieves a superpolynomial convergence rate. Together, these two observations suggest that hybrid schemes~\cite{Djurdjevac2025Hybrid}, combining high order approximations in high density regions with particle-based or other positivity-preserving approximations in low density regions, may be the most efficient at faithfully capturing the dynamics of the particle system.

For the rest of the introduction, we give a more technical description of our main results. Our first main result, Theorem \ref{thm:main-nonlocal-testfunction}, gives a uniform-in-time weak error bound measured through the Laplace transform,
\[
|\E [e^{\langle u_t, \varphi \rangle}] - \E [e^{\langle \mu_t, \varphi \rangle}] | 
\]
where  $(\mu_t)_{t \ge 0}$ is the empirical distribution of Brownian particles, $\varphi$ is a complex valued test function in $C^{\alpha}$, $\alpha >0$, under a smallness assumption relating $N$ and $\varepsilon$ ensuring that the particle number is large compared to the spatial resolution. The use of the Laplace transform is convenient for two reasons. First, it captures the probability law of $u_t$ and by Cauchy's integral formula, the weak error for the Laplace transform also controls the weak error for moments of all orders, see Corollary \ref{cor:moments}. Second, the Laplace transform is naturally connected to the duality of the empirical measure with the Hamilton--Jacobi--Bellman equation. In particular, for the terminal condition $\mu\mapsto e^{\langle \mu, \varphi\rangle}$, the duality of~\cite{Konarovskyi2019} provides an explicit solution to the infinite-dimensional Kolmogorov backward equation associated to the empirical distribution of independent Brownian particles, which is very convenient for the analysis. It would not be necessary to work with this explicit solution, and as in~\cite{Djurdjevac2026} we could also consider an approximate solution to the infinite-dimensional Kolmogorov equation, which would apply to interacting particle systems with more complex dynamics, but this would complicate the analysis and to a certain extent obscure the main message of the paper. In any case, the emphasis on the infinite-dimensional Markovian structure behind the Dean--Kawasaki equation distinguishes our analysis from~\cite{Cornalba2023, Cornalba2023Density} who take a finite-dimensional perspective based on cylinder functions. Let us also comment that for simplicity we restrict to the one-point distribution of $\mu_t$. But with minor additional effort it would be possible to consider a path Laplace transform, see Remark~\ref{rmk:multiple-times}.

The resulting error bound has three contributions: the initial approximation error, the Fourier cut-off error, and the error produced by the negative part of $u$.  In the high-density regime, when $\rho_0 \geq \rho_{\min}>0$, the contribution of the negative part is exponentially small. As a consequence, the spectral approximation achieves a superpolynomial weak convergence rate if $\varphi \in C^\infty$: For $f(x)=\sqrt{x^+}$, the dynamic part of the weak error is of order
\begin{equation}\label{eq:bound-high-density}
  O\!\bigl(N^{-1-\alpha/d}(\log N)^{3\alpha/d}\bigr)
  \qquad\text{as }N\to\infty,
\end{equation}
for \emph{all} $\alpha>0$, provided that we choose $\varepsilon=N^{-1/d}(\log N)^{3/d}$.

The situation changes sharply in the low density regime. If $\rho_{\min} = 0$, the exponential suppression of the negative part disappears, and the upper bound deteriorates to
\begin{equation}\label{eq:bound-low-density}
  O\!\bigl(N^{-1}\varepsilon^\alpha + N^{-3/2}\varepsilon^{-d/2}\bigr)
  =
  O\!\bigl(N^{-(3\alpha + d)/(2\alpha+d)}\bigr) 
\end{equation}
for the optimal choice $\varepsilon = N^{-1/(2\alpha+d)}$. We emphasize that the bounds~\eqref{eq:bound-high-density} and~\eqref{eq:bound-low-density} are both uniform in time, which we achieve by exploiting the spectral gap of the Laplacian on $\mathbb T^d$. To the best of our knowledge, this is the first such uniform approximation result for Dean--Kawasaki equations.

For smooth test functions the bound~\eqref{eq:bound-low-density} is far worse than the superpolynomial rate available when $\rho_{\min}>0$. And for $d \in \{1,2\}$ it is also worse than that of the positivity preserving SPDE approximation suggested in \cite{Djurdjevac2024}. While here we compare two upper bounds, we will later comment on lower bounds and discuss that the loss of positivity preservation indeed limits the approximation quality in the low density regime.

In the previous results, we allowed the initial density $\rho_0$ to have large concentrations.
To do this, we mainly controlled the $L^1$ norm of $u_t$, and therefore we needed uniform control of the test function $\varphi$. In the next result, Theorem \ref{thm:main-local-testfunction}, we take a different perspective: We assume that $\rho_{\max}$ is not too large, which allows us to get good control of the $L^\infty$ norm of $u_t$. In the duality $\langle u_t, \varphi\rangle$ we can therefore estimate $\varphi$ in $L^1$ based Besov spaces, which allows us to consider spatially localized test functions. This seems to be a new observation in this context.

The last main result shows that the condition $\rho_{\min}>0$ is essential for superpolynomial convergence rates.
We give an example with $\rho_{\min}=0$ and take $f=\sqrt{\cdot}$,  on the one dimensional torus, and show that the SPDE~\eqref{eq:approx_DK} only has a limited approximation rate. The reason is that the negative part of the solution is relatively large, at least for a short period of time before the heat flow starts to dominate the fluctuations. More precisely, we prove a lower bound for the approximation of the variance:
\begin{equation}\label{eq:lower-bound-intro}
\left|
\E\left[\left(\langle u_t,\varphi\rangle
- \langle p_t \ast u_0,\varphi\rangle\right)^2\right]
-
\E\left[\left(\langle \mu_t,\varphi\rangle
- \langle p_t \ast \mu_0,\varphi\rangle\right)^2\right]
\right|
\gtrsim
\frac{e^{-4\pi^2 t}}{N^{3/2}}
\frac{\varepsilon^{5/2}}{|\log \varepsilon|^{3}}
\end{equation}
as $N\to\infty$, whenever $\varepsilon=\varepsilon(N)\gtrsim N^{-a}$ for some $a>0$. Here  $p_t$ denotes the heat kernel generated by $\tfrac12 \Delta$ and we note that $\E[u_t] = p_t \ast u_0$ and, in the sense of integrating against test functions, $\E[\mu_t] = p_t \ast \mu_0$.

Note that \eqref{eq:lower-bound-intro} does not match the upper bound
$O(N^{-3/2}\varepsilon^{-d/2})$  from~\eqref{eq:bound-low-density}, but it shows that the weak error cannot decay superpolynomially unless $\rho_{\min}>0$. In the simple example of independent Brownian particles this problem can only arise if the initial condition is too small (or even zero) in some regions, basically because the heat equation satisfies a maximum principle and the solution at later times is bounded from below by $\rho_{\min}$. But for interacting particles or even just particles in a potential the Fokker--Planck equation may violate the maximum principle and the dynamics themselves may produce low density regions~\cite{Wehlitz2025}, and therefore it is crucial to understand approximations in the low density case. %

The paper is organised as follows. Section \ref{sec:main} introduces the setting and notation, states the well-posedness result for $u$ and proves the main results. Section \ref{sec:lower_bound} establishes the lower bound on the negative part of $u$ in one space dimension for densities that are not strictly positive. In Section \ref{sec:negative} we provide upper bounds for the negative part of $u$. Section \ref{sec:initial_approx} gives the error in the initial approximation of particles by a density. 
Section \ref{sec:HJB} treats the duality and the Cole--Hopf solution for the complex-valued Hamilton--Jacobi--Bellman equation.
The appendices contain the well-posedness proof and auxiliary results. Section~\ref{sec:numerics} presents numerical experiments.


\paragraph{Notation} 
Throughout, let
\[
    B^\alpha_{p,q} := B^\alpha_{p,q}(\T^d),\qquad \alpha\in \R,p,q\in [1,\infty],
\]
denote the closure of $C^\infty(\T^d)$ with respect to the Hölder-Besov norm $\|\cdot \|_{B^\alpha_{p,q}}$, see \cite{Bahouri2011} for Besov spaces on $\R^d$, and \cite{Schmeisser1987, Gubinelli2015EBP} for Besov spaces on $\T^d$. For $p=q=\infty$ we write
\[
    C^\alpha:=B^\alpha_{\infty,\infty}.
\]
For complex-valued functions $\varphi$ we write
\[
    \varphi \in B^{\alpha}_{p,q}(\T^d, \C),\qquad \text{or} \qquad \varphi \in C^\alpha(\T^d, \C),
\]
if both $\operatorname{Re}(\varphi)$ and $\operatorname{Im}(\varphi)$ are in the corresponding function space, and in that case
\[
    \|\varphi\|_{B^\alpha_{p,q}} := \|\operatorname{Re}(\varphi)\|_{B^\alpha_{p,q}} + \|\operatorname{Im}(\varphi)\|_{B^\alpha_{p,q}}.
\]

\section{Main results}\label{sec:main}

Let $N \in \N$ and let $\mu_t = \frac{1}{N} \sum_{i = 1}^N \delta_{B^i_t}$, $t \geq 0$, be the empirical measure
of $N$ independent Brownian motions $(B^i)_{i = 1,\dots,N}$. By \cite{Dean1996, Konarovskyi2019} the process $\mu$ is a martingale solution of the Dean--Kawasaki equation
\begin{equation*}
    \mathd \mu_t = \frac12 \Delta \mu_t \mathd t + \frac{1}{\sqrt N} \nabla \cdot (\sqrt{\mu_t} \mathd W_t),
\end{equation*}
where $W$ is a cylindrical vector-valued Wiener process with identity covariance on  $L^2 (\mathbb{T}^d; \mathbb{R}^d)$. We assume that the initial conditions $(B^i_0)_{i=1,\dots, N}$ are i.i.d. with common density $\rho_0 \in L^\infty(\mathbb T^d)$. We write 
\begin{equation*}
    \rho_{\min} := \operatorname{essinf}_{x \in \mathbb T^d} \rho_0(x) \geq 0,\qquad \rho_{\max}:=\|\rho_0\|_{L^\infty}.
\end{equation*}
Our goal is to model the statistical behavior of the process $\mu$ with a continuum model, for which we choose an SPDE with spatially smooth trajectories. Concretely, we consider the regularized Dean--Kawasaki equation
\begin{equation}\label{eq:reg_DK}
     \mathd u_t = \frac{1}{2} \Delta u_t \mathd t + \frac{1}{\sqrt{N}} \nabla
   \cdot \Bigl(K_{\varepsilon} \ast \left(f (u_t^+) \mathd W_t\right) \Bigr)
\end{equation}
on $\mathbb{R}_+ \times \mathbb{T}^d$, for the unit torus $\T = \R/\Z$, where $u^+ := \max \{ u,0 \}$ is the positive part of
$u$, the Wiener process $W$ is as above, the mollifier $K_\varepsilon$ is given by
\[ 
K_{\varepsilon} =\mathcal{F}^{- 1} (\chi (\varepsilon\cdot)),
\]
with an even function $\chi \in C^{\infty}_c(\R^d)$ satisfying $\chi \equiv 1$ in a
neighborhood of $0$, and $f:\R_+\to\R$ is a continuous function with $f(x) \lesssim \sqrt{x}$.

\begin{definition}[Mild and weak solution]\label{def:weak_sol}
Let $u_0$ be an $L^2(\mathbb T^d)$-valued random variable.
An $(\mathcal F_t)_{t\geq 0}$-adapted and continuous $L^2(\mathbb T^d)$-valued process
$u=(u_t)_{t\geq 0}$, is called a mild solution to \eqref{eq:reg_DK} with initial condition $u_0$, if for every $t\geq 0$, almost surely,
\begin{equation}\label{eq:mild_formulation}
    u_t
    =
    p_t * u_0
    +
    \frac{1}{\sqrt N}\int_0^t
    p_{t-s} * \nabla\cdot\Bigl(K_\varepsilon * \bigl(f(u_s^+)\,\mathrm{d}W_s\bigr)\Bigr),
\end{equation}
where $p_t$ denotes the heat kernel on $\mathbb T^d$. It is called a weak solution if, 
for every test function
$\varphi\in C^\infty(\mathbb T^d)$  and every $t\geq 0$, almost surely,
\begin{equation}\label{eq:weak_formulation}
    \langle u_t,\varphi\rangle
    =
    \langle u_0,\varphi\rangle
    +
    \frac12 \int_0^t \langle u_s,\Delta\varphi\rangle\,\mathrm{d}s
    +
    \frac{1}{\sqrt N}\int_0^t
    \big\langle \nabla\cdot\bigl(K_\varepsilon * (f(u_s^+)\,\mathrm{d}W_s)\bigr),\varphi\big\rangle.
\end{equation}
\end{definition}

Using integration by parts, we can write \eqref{eq:weak_formulation} as 
\begin{equation}\label{eq:weak_formulation_ibp}
    \langle u_t,\varphi\rangle
    =
    \langle u_0,\varphi\rangle
    +
    \frac12 \int_0^t \langle u_s,\Delta\varphi\rangle\,\mathrm{d}s
    -
    \frac{1}{\sqrt N}\int_0^t
    \big\langle K_\varepsilon * (f(u_s^+)\,\mathrm{d}W_s),\nabla\varphi\big\rangle.
\end{equation}

\begin{proposition}[Well-posedness and weak existence for the regularized equation]\label{prop:wp}

\leavevmode\par
\begin{enumerate}

\item[i)] Assume that $f$ is globally Lipschitz. Then, for every $u_0\in L^2(\Omega;L^2(\mathbb T^d))$  that is  independent of $W$, there exists a pathwise unique  mild solution $u$ to the regularized Dean--Kawasaki equation \eqref{eq:reg_DK}. Moreover, $(u_t)_{t \in [0,T]}\in L^2\bigl(\Omega;C([0,T];L^2(\mathbb T^d))\bigr)$ for each $T>0$, and $u$ is also a weak solution. 

\item[ii)] Assume  that $f$ is continuous and has at most linear growth. 
Then, for every $u_0 \in L^2(\Omega;L^2(\mathbb T^d))$, there exists a
probabilistically weak mild solution to the regularized equation \eqref{eq:reg_DK}. More precisely, there exist a stochastic basis $(\tilde\Omega,\tilde{\mathcal F},
(\tilde{\mathcal F}_t)_{t\ge0},\tilde{\mathbb P})$ with a cylindrical Wiener process $\tilde W$ on $L^2(\mathbb{T}^d;\mathbb{R}^d)$, and an
$(\tilde{\mathcal F}_t)_{t\ge0}$-adapted process $u$ satisfying $(u_t)_{t \in [0,T]}\in L^2\bigl(\tilde\Omega;C([0,T];L^2(\mathbb T^d))\bigr)$ for each $T>0$, 
such that $u$ is a mild and a weak solution driven by $\tilde{W}$.
\end{enumerate}
\end{proposition}

The proof of the proposition can be found in Appendix~\ref{Appendix_wp}. The second part of the statement applies to $f(x) = \sqrt{x}$.

Although the regularized equation does not preserve positivity and
therefore the $L^1$ norm is not conserved in general, the total mass in the
signed sense is still conserved. Indeed, applying \eqref{eq:weak_formulation_ibp} to the
constant function $\varphi\equiv 1$ gives
\[
    \int_{\mathbb{T}^d} u_t(x)\,\mathd x
    =
    \int_{\mathbb{T}^d} u_0(x)\,\mathd x, \qquad \text{ for all } t\geq 0.
\]
We decompose \eqref{eq:mild_formulation} into expectation and fluctuations,
\begin{equation}\label{eq:v-def}
    u_t = p_t \ast u_0 + v_t,
\end{equation}
where the fluctuations are given by
\begin{equation}\label{def:v}
    v_t := \frac{1}{\sqrt N}\int_0^t
    p_{t-s} * \nabla\cdot\Bigl(K_\varepsilon * \bigl(f(u_s^+)\,\mathrm{d}W_s\bigr)\Bigr).
\end{equation}
The following parameters govern the size of the fluctuations $v$:
\begin{equation}
    \delta \assign \varepsilon^{- d / 2} N^{- 1 / 2}, \qquad \kappa \assign
     \delta \sqrt{1\vee d \log \frac{1}{\varepsilon}} .
\end{equation}
More precisely, if $u_0 \equiv 1$, for small times the standard deviation of the stochastic integral in~\eqref{def:v} is of order $\delta \sqrt{t}$ for each $x \in \T^d$,  and the uniform norm in $x$ is of order $\kappa \sqrt{t}$.

\begin{assumption}[Fluctuation scaling]\label{ass:kappa}
    Throughout we make the standing assumption that \(N\) and \(\varepsilon\) are chosen such that $\delta \leq \kappa \leq 1$.
\end{assumption}
This condition is equivalent to 
\begin{equation}\label{kappa_condition}
    N \geq \varepsilon^{-d} \big(1\vee d\log \frac{1}{\varepsilon}\big) ,
\end{equation}
so it encodes that the number of particles is sufficiently large relative to the mollification scale $\varepsilon$. This will guarantee that the fluctuations $v_t$ do not dominate the expectation $p_t \ast u_0$, and it is crucial to obtain a control of the negative part of $u$. Indeed, $p_t \ast u_0$ preserves pointwise the lower bound of the initial condition, while $v_t$ may take negative values. So, for $u_t$ to remain positive we need $v_t$ to be sufficiently small compared to $p_t \ast u_0$. Note also that for $N \lesssim \varepsilon^{-d}$ the complexity of the SPDE model is higher than that of the particle system, because the SPDE has $O(\varepsilon^{-d})$ Fourier modes and there are $N$ particles.

Our first main result controls the difference of the moment generating functions of $u_t$ and $\mu_t$, both extended to a complex neighborhood of the origin. While the dynamics of $u$ approximate the dynamics of $\mu$ very well, we also need to approximate the particle initial condition $\mu_0$ to high order by a density $u_0$. A natural idea would be to take $K_\varepsilon\ast \mu_0$, but this will take negative values and  not be a probability density. Therefore, we start $u$ from $\rho_0 + K_\varepsilon\ast (\mu_0 - \rho_0)$, corrected if this gets too close to $0$.

\begin{theorem}\label{thm:main-nonlocal-testfunction}
  Assume that $f \in C(\R_+,\R)$ is such that $f (x) \lesssim \sqrt{x}$. Let $\nu := K_\varepsilon *(\mu_0-\rho_0)$ and then
  \[
    \bar{u} := \begin{cases}
        \rho_0 + \nu,
        & \|\nu\|_\infty \leq \frac{\rho_{\min}}{2},\\
        \rho_0,
        & \text{otherwise}.
        \end{cases}
  \]
  Let $u$ be a mild and weak solution to~\eqref{eq:reg_DK}  with initial condition $u_0$ such that $\operatorname{law}(u_0) = \operatorname{law}(\bar u)$. Then there exist $c > 0$ and $\gamma \in (0,\pi/4]$ such that for any $\alpha > 0$ and any $\varphi \in C^{\alpha}(\T^d, \C)$ with $\nabla \varphi \in L^\infty$ and $\kappa^2 \| \varphi\|_\infty \leq \gamma$
  \begin{align} \label{eq:main}
    &| \mathbb{E} [e^{\langle u_t, \varphi \rangle}] -\mathbb{E} [e^{\langle
    \mu_t, \varphi \rangle}] |  \lesssim_{\alpha} \frac{e^{C\|\varphi\|_\infty}}{N^{1/2}}\left(\varepsilon^\alpha \|\varphi\|_{C^\alpha}
        + \|\varphi\|_{L^\infty} 
        \exp\left(-c \frac{\rho_{\min}^2}{\kappa^2\rho_{\max}}\right)\right)\\ \nonumber
    & \hspace{10pt}+ \frac{e^{C (\| \varphi \|_{\infty} + \kappa^2 \rho_{\max} \| \varphi \|_{\infty}^2 )}}{N}
    \left(\varepsilon^{\alpha} \|\nabla \varphi\|_\infty \| \varphi \|_{C^{\alpha}} + (\| \operatorname{id} - f^2 \|_{\infty} + \delta \rho_{\max}^{1/2} \exp\left(-c \frac{\rho_{\min}^2}{\kappa^2\rho_{\max}}\right)) \| \nabla \varphi \|_{\infty}^2\right). 
  \end{align}
\end{theorem}

The first term on the right hand side corresponds to the approximation error made at the initial condition, while the second one captures the weak error from the dynamics. For $\rho_{\min}=0$, the weak error from the initial condition can be improved, see Remark~\ref{rmk:zero-IC} following  the proof.

\begin{proof}
  By definition of $C^{\alpha}$ as the closure of smooth functions in the $B^{\alpha}_{\infty,\infty}$ norm, we may assume that $\varphi \in C^\infty(\T^d,\C)$. For real-valued twice continuously differentiable $\varphi$ we know from~\cite{Konarovskyi2019} the duality $\E[e^{\langle\mu_t, \varphi \rangle}] = \E[e^{\langle
    \mu_0, \varphi_0 \rangle}]$, where $\varphi$ solves the Hamilton--Jacobi--Bellman equation $\partial_s \varphi + \frac{1}{2} \Delta \varphi + \frac{1}{2 N} (
  \nabla \varphi )^2 = 0$ on $[0, t] \times \T^d$, with terminal condition $\varphi_t = \varphi$. We discuss in Section~\ref{sec:HJB} the extension of this duality result to complex-valued $\varphi$ with $\|\tfrac1N \varphi\|_\infty \leq \tfrac\pi4$, which is satisfied since $\tfrac1N \|\varphi\|_\infty \leq \kappa^2 \|\varphi\|_\infty \leq \gamma \leq \tfrac\pi4$. We also discuss the solution of the complex-valued Hamilton--Jacobi--Bellman equation by the Cole--Hopf transform $\varphi_s = N \log(p_{t-s}\ast \exp(\tfrac1N \varphi))$. Lemma~\ref{lem:CH-regularity} collects regularity results for $(\varphi_s)_{s \in [0,t]}$ that are uniform in $s$ and $t$; note that the lemma is formulated for the equation in forward time, so that here we have to reverse time.
  
   By the duality, 
   we have
  \begin{equation}\label{eq:main-pr1}
    | \mathbb{E} [e^{\langle u_t, \varphi \rangle}] -\mathbb{E} [e^{\langle
    \mu_t, \varphi \rangle}] | \leq |\mathbb E[e^{\langle u_t,\varphi\rangle}] - \mathbb E[e^{\langle u_0, \varphi_0\rangle}]| + |\mathbb E[e^{\langle u_0,\varphi_0\rangle}] - \mathbb E[e^{\langle \mu_0,\varphi_0\rangle}]|.
  \end{equation}
  For the second term, which corresponds to the error from the initial approximation, we use that $\operatorname{law}(u_0)=\operatorname{law}(\bar u)$ to replace $u_0$ in the expectation by $\bar u$ and then we apply Proposition~\ref{prop:IC} 
  to obtain
  \begin{align*}
    |\mathbb E[e^{\langle u_0,\varphi_0\rangle}] - \mathbb E[e^{\langle \mu_0,\varphi_0\rangle}]| & \leq \mathbb E[e^{C\|\varphi\|_\infty}|\langle \bar u - \mu_0, \varphi_0\rangle|] \\
    & \lesssim e^{C \|\varphi\|_\infty} \left(N^{-1/2}\varepsilon^\alpha \|\varphi\|_{C^{\alpha}}
        + N^{-1/2}\|\varphi\|_{L^\infty}
        \exp\left(-c \frac{\rho_{\min}^2}{\kappa^2\rho_{\max}}\right)\right),
  \end{align*}
  where we applied~\eqref{eq:CH-Lebesgue} and~\eqref{eq:CH-Besov} from Lemma~\ref{lem:CH-regularity} in the appendix to bound $\|\varphi_0\|_\infty \leq C \|\varphi\|_\infty$ and $\|\varphi_0\|_{C^\alpha} \lesssim \|\varphi\|_{C^{\alpha}}$, both uniformly in the terminal time $t$.

  To bound the first term in~\eqref{eq:main-pr1}, we use that $\langle u_t, \varphi\rangle = \langle u_t, \operatorname{Re}(\varphi)\rangle + i \langle u_t, \operatorname{Im}(\varphi)\rangle$, so we can apply the weak formulation  \eqref{eq:weak_formulation_ibp} to real and imaginary part of $\varphi$ separately.
  Therefore
  \begin{align} \nonumber
    | \mathbb{E} [e^{\langle u_t, \varphi \rangle}] -\mathbb{E} [e^{\langle
    u_0, \varphi_0 \rangle}] | & \leq \int_0^t \left| \mathbb{E} \left[
    e^{\langle u_s, \varphi_s \rangle} \left( \langle u_s, \partial_s
    \varphi_s \rangle + \left\langle u_s, \frac{1}{2} \Delta \varphi_s
    \right\rangle \right.\right.\right.\\ \nonumber
    &\hspace{100pt} \left.\left.\left. + \frac{1}{2 N} \langle ( K_{\varepsilon} \ast \nabla
    \varphi_s )^2, f (u^+_s)^2 \rangle \right) \right] \right| \mathd s\\ \nonumber
    & \leq \int_0^t \left| \mathbb{E} \left[ e^{\langle u_s, \varphi_s
    \rangle} \left( \frac{1}{2 N} \langle ( \nabla \varphi_s )^2 - (
    K_{\varepsilon} \ast \nabla \varphi_s )^2, u_s \rangle \right) \right]
    \right| \mathd s\\ \nonumber
    & \quad + \int_0^t \left| \mathbb{E} \left[ e^{\langle u_s, \varphi_s
    \rangle} \left( \frac{1}{2 N} \langle ( K_{\varepsilon} \ast \nabla
    \varphi_s )^2, u_s^+ - f (u^+_s)^2 \rangle \right) \right] \right| \mathd
    s\\ \nonumber
    & \quad + \int_0^t \left| \mathbb{E} \left[ e^{\langle u_s, \varphi_s
    \rangle} \left( \frac{1}{2 N} \langle ( K_{\varepsilon} \ast \nabla
    \varphi_s )^2, u_s^- \rangle \right) \right] \right| \mathd s\\ \nonumber
    & \lesssim \frac{1}{N} \int_0^t \mathbb{E} [|e^{2 \langle u_s, \varphi_s
    \rangle}|]^{1 / 2} \| ( \nabla \varphi_s )^2 - ( K_{\varepsilon} \ast
    \nabla \varphi_s )^2 \|_{\infty} \mathbb{E} [\| u_s \|_{L^1}^2]^{1 / 2}
    \mathd s\\ \nonumber
    & \quad + \frac{1}{N} \int_0^t \mathbb{E} [|e^{\langle u_s, \varphi_s
    \rangle}|] \| \nabla \varphi_s \|_{\infty}^2 \| \operatorname{id} - f^2 \|_{\infty}
    \mathd s\\ 
    & \quad + \frac{1}{N} \int_0^t \mathbb{E} [|e^{2 \langle u_s, \varphi_s
    \rangle}|]^{1 / 2} \| \nabla \varphi_s \|_{\infty}^2 \mathbb{E} [\| u_s^-
    \|_{L^1}^2]^{1 / 2} \mathd s. \label{eq:main-pr2-s2}
  \end{align}
  We start by estimating the factor $\E[|e^{\lambda \langle u_s,\varphi_s\rangle}|]$, for $\lambda \in \{1,2\}$, which appears in each of the three terms in~\eqref{eq:main-pr2-s2}. From the bound~\eqref{eq:CH-Lebesgue} for $\varphi$ and from the bound on the exponential moments of $\|u\|_{L^1}$ from Corollary~\ref{cor:exp-moments-u} in Section~\ref{sec:negative} below, we know that for $\lambda \in \{ 1, 2 \}$ and for $\kappa^2\|\varphi\|_\infty \leq \gamma$ with $\gamma \in (0,\tfrac\pi4]$ sufficiently small, the following bound holds uniformly in time, with a constant $C>0$ that may change in each step:
  \begin{equation*}
    \mathbb{E} [|e^{\lambda \langle u_s, \varphi_s \rangle}|] \leq
     \mathbb{E} \left[ e^{C\lambda  \| \varphi \|_{\infty} \| u_s \|_{L^1}}
     \right] \lesssim \exp (C\lambda \| \varphi \|_{\infty}
     +C\lambda^2 \| \varphi \|_{\infty}^2 \kappa^2 \rho_{\max})  \leq e^{C (\| \varphi \|_{\infty} + \| \varphi \|_{\infty}^2 \kappa^2\rho_{\max})},
  \end{equation*}
  and therefore~\eqref{eq:main-pr2-s2} simplifies to
  \begin{align} \nonumber
      & | \mathbb{E} [e^{\langle u_t, \varphi \rangle}] -\mathbb{E} [e^{\langle
    u_0, \varphi_0 \rangle}] |\\ \label{eq:main-pr3-s2}
      & \lesssim \frac{e^{C (\| \varphi \|_{\infty} + \| \varphi \|_{\infty}^2 \kappa^2 \rho_{\max})}}{N} \left(\int_0^t \| (\nabla \varphi_s)^2 - ( K_{\varepsilon} \ast
    \nabla \varphi_s )^2 \|_{\infty} \mathbb{E} [\| u_s \|_{L^1}^2]^{1 / 2}
    \mathd s\right. \\ \label{eq:main-pr4-s2}
      & \hspace{130pt} +  \int_0^t \| \nabla \varphi_s \|_{\infty}^2 \| \operatorname{id} - f^2 \|_{\infty} \mathd s \\ \label{eq:main-pr5-s2}
      & \hspace{130pt} + \left.\int_0^t \| \nabla \varphi_s \|_{\infty}^2 \mathbb{E} [\| u_s^- \|_{L^1}^2]^{1 / 2} \mathd s\right).
  \end{align}
  
  We proceed by estimating each of the three terms on the right hand side separately. The first term corresponds to the approximation error made by convolution with $K_\varepsilon$, the second one corresponds to the error made by replacing the square root by $f$ and it vanishes for $f(x) = \sqrt{x}$, and the third one corresponds to the error from negative values of $u$, which cannot be ruled out because our approximation violates the maximum principle.
  
  To estimate~\eqref{eq:main-pr3-s2}, we bound for $\alpha > 0$
  \begin{align*}
    \| ( \nabla \varphi_s )^2 - ( K_{\varepsilon} \ast \nabla \varphi_s )^2
    \|_{\infty} & = \| (\nabla \varphi_s + \nabla K_{\varepsilon} \ast
    \varphi_s) \cdot (\nabla \varphi_s - \nabla K_{\varepsilon} \ast
    \varphi_s) \|_{\infty}\\
    & \lesssim \| \nabla \varphi_s \|_{\infty} \varepsilon^{\alpha} \|\nabla
    \varphi_s \|_{C^{\alpha}} \lesssim e^{-2c(t-s)}(t-s)^{-1/2} \| \nabla \varphi \|_{\infty} \varepsilon^{\alpha} \|
    \varphi \|_{C^{\alpha}},
  \end{align*}  
  where we applied Lemma~\ref{lem:K-eps-error} to control the approximation error from $K_\varepsilon$, and Lemma~\ref{lem:CH-regularity} to control the regularity of $\varphi_s$ uniformly in time, with exponential decay due to the spectral gap of the heat operator on $\T^d$. For the factor $\E[\| u_s \|_{L^1}^2]^{1 / 2}$ we  split $u_s = p_s\ast u_0 + v_s$, where $v_s$ is defined by the equality, and we apply Lemma~\ref{lem:v-Lebesgue-moments} to bound $v_s$:
  \[
     \sup_{s \geq 0} \mathbb{E} [\| u_s \|_{L^1}^2]^{1 / 2} \leq \sup_{s \geq 0} \mathbb{E} [\| p_s\ast u_0 \|_{L^1}^2]^{1 / 2} + \sup_{s \geq 0} \mathbb{E} [\| v_s \|_{L^2}^2]^{1 / 2} \lesssim 1 + \delta \E[\|u_0\|_{L^1}^{1/2}] \lesssim 1,
  \]
  where we used that $p_s\ast u_0$ is a probability density for all $s \geq 0$, and that $\delta \leq 1$. Therefore,
  \begin{align*}
    &\int_0^t \| (\nabla \varphi_s)^2 - ( K_{\varepsilon} \ast
    \nabla \varphi_s )^2 \|_{\infty} \mathbb{E} [\| u_s \|_{L^1}^2]^{1 / 2}
    \mathd s \\
    &\hspace{40pt} \lesssim \int_0^t \frac{e^{-2c(t-s)}}{(t-s)^{1/2}} \| \nabla \varphi \|_{\infty} \varepsilon^{\alpha} \|
    \varphi \|_{C^{\alpha}} ( 1 + \delta \E[\|u_0\|_{L^1}^{1/2}])\mathd s \\
    &\hspace{40pt} \lesssim \| \nabla \varphi \|_{\infty} \varepsilon^{\alpha} \|
    \varphi \|_{C^{\alpha}} ( 1 + \delta \E[\|u_0\|_{L^1}^{1/2}]),
  \end{align*}
  uniformly in $t\geq 0$, which concludes the bound for~\eqref{eq:main-pr3-s2}.

  For the term~\eqref{eq:main-pr4-s2} it suffices to note that $\|\nabla \varphi_s\|_\infty \lesssim e^{-c(t-s)}\|\nabla \varphi\|_\infty$ by Lemma~
  \ref{lem:CH-regularity},  and therefore the time integral is again bounded uniformly in $t \geq 0$.

  It remains to bound the contribution from the negative part of $u$ in~\eqref{eq:main-pr5-s2}. Since $u_0 \in [\rho_{\min}/2,\rho_{\max} + \rho_{\min}/2] \subset [\rho_{\min}/2, 2\rho_{\max}]$  
  by construction, we obtain from~\eqref{eq:moments-negative-u} below, with $p=1$ and $r=2$ and $c>0$ that changes its value in the second inequality, and for $u_{\min} = \min_x u_0(x)$:
  \[
    \E[\|u_s^-\|_{L^1}^2]^{1/2} \lesssim \delta \E\left[\|u_0\|_{L^2}^{1/2}  \exp \left( -\frac{c}{2} \frac{u_{\min}^2}{\kappa^2 \|u_0\|_\infty}\right)\right] \lesssim \delta \rho_{\max}^{1/2}\exp \left( -c \frac{\rho_{\min}^2}{\kappa^2 \rho_{\max}}\right),
  \]
  which together with the estimate $\|\nabla \varphi_s\|_\infty \lesssim e^{-c(t-s)}\|\nabla \varphi\|_\infty$ concludes the proof.
\end{proof}

\begin{remark}\label{rmk:zero-IC}
  The first term in our estimate~\eqref{eq:main} captures the initial error by approximating $\mu_0$ with $u_0$. If $\rho_{\min}=0$, then $u_0 = \rho_0$ and we can improve this term as follows:
  \begin{align*}
    |\mathbb E[e^{\langle \rho_0,\varphi_0\rangle}] - \mathbb E[e^{\langle \mu_0,\varphi_0\rangle}]| & \leq |\mathbb E[e^{\langle \rho_0,\varphi_0\rangle}\langle \rho_0 - \mu_0, \varphi_0\rangle]| + \mathbb E[e^{\|\varphi_0\|_\infty}\langle \rho_0 - \mu_0, \varphi_0\rangle^2]\\
    & \leq 0 + e^{\|\varphi\|_\infty} N^{-1} \|\varphi\|_\infty^2,
  \end{align*}
 so that we obtain order $N^{-1}$ in that case.
\end{remark}

In the next remark we focus on the dynamic part of the weak error and distinguish the two cases $\rho_{\min} = 0$  and $\rho_{\min}>0$.

\begin{remark}
Assume that $f(x) = \sqrt{x}$, $x \in \R_+$.
\begin{enumerate}
    \item For $\rho_{\min} = 0$, the dynamic part of the weak error becomes
    \[
        \frac{e^{C (\| \varphi \|_{\infty} + \| \varphi \|_{\infty}^2 \kappa^2)}}{N}
    \left(\varepsilon^{\alpha} \| \varphi \|_{C^{\alpha}}^2 + \delta \rho_{\max}^{1/2} \| \nabla \varphi \|_{\infty}^2\right) = O(N^{-1}\varepsilon^\alpha + N^{-3/2} \varepsilon^{-d/2}).
    \]
    For given $\alpha>0$ this is optimized by $\varepsilon = N^{-\frac{1}{2\alpha+d}}$ for which we obtain an error of order 
    \[
       O(N^{-\frac{3\alpha+d}{2\alpha+d}})
       \geq O(N^{-3/2}),
    \]
   which incidentally for $d=1,2$ is worse than $O(N^{-1-\tfrac{2}{d+2}} \log N)$, the weak error that the positivity preserving SPDE approximation of \cite{Djurdjevac2024} achieves.
    This is only an upper bound for the weak error. But in the next section we will derive lower bounds which indicate that for $\rho_{\min}=0$ the approximation rate is indeed limited.
    \item For $\rho_{\min} > 0$ and $\varepsilon<1/e$, the dynamic part of the weak error becomes
    \begin{align*}
        & \frac{e^{C (\| \varphi \|_{\infty} + \| \varphi \|_{\infty}^2 \kappa^2 \rho_{\max})}}{N}
    \left(\varepsilon^{\alpha} \| \varphi \|_{C^{\alpha}}^2 + \delta \rho_{\max}^{1/2} \exp\left(-c \frac{\rho_{\min}^2}{\kappa^2\rho_{\max}}\right) \| \nabla \varphi \|_{\infty}^2\right) \\
        &\hspace{30pt}\leq C_{\varphi,\alpha,\rho_0}\left( N^{-1} \varepsilon^\alpha + N^{-3/2} \varepsilon^{-d/2}\exp\left(-c_{\rho_0,d} \frac{N \varepsilon^d}{\log(1/\varepsilon)} \right)\right),
    \end{align*}
    and with $\varepsilon = N^{-1/d} (\log N)^{3/d}$ this becomes, with changing value of $c_{\rho_0,d}$, 
    \begin{align*}
        & O\left(N^{-1-\alpha/d}(\log N)^{3\alpha/d} + N^{-1} (\log N)^{-3/2}\exp\left(-c_{\rho_0,d} \frac{(\log N)^{3}}{\log N + \log \log N}\right) \right) \\
        & \hspace{30pt} \leq O\left(N^{-1-\alpha/d} (\log N)^{3\alpha/d} + N^{-1} \exp\left(-c_{\rho_0,d} (\log N)^{2} \right) \right) \\
        & \hspace{30pt} \leq O\left(N^{-1-\alpha/d} (\log N)^{3\alpha/d}\right),
    \end{align*}
    as $N \to \infty$. This shows that on smooth test functions $\varphi \in C^\infty(\T^d,\C)$ the SPDE~\eqref{eq:reg_DK} weakly approximates the empirical distribution of Brownian particles at a superpolynomial order -- but only under the assumption that $\rho_{\min}>0$.
\end{enumerate}
\end{remark}

    \begin{remark}\label{rmk:multiple-times}
    For simplicity we focus here on the weak error in the Laplace transform at a fixed time. But with minor modifications, the same proof allows us to control a Laplace transform on path space: For $0 \leq t_1 < t_2 < \dots < t_m = T$ and $\varphi_1, \dots, \varphi_m \in C^\alpha(\T^d,\C)$ and $g \in L^\infty([0,T], C^\alpha(\T^d,\C))$ we could use the Markov property between times $t_k$ and $t_{k+1}$ to control the weak error
    \[
        \left|\E \left[\exp\left(\sum_{i=1}^m \langle u_{t_i}, \varphi_i \rangle + \int_0^T \langle u_s, g_s \rangle \mathd s\right)\right] - \E \left[\exp\left(\sum_{i=1}^m \langle \mu_{t_i}, \varphi_i \rangle + \int_0^T \langle \mu_s, g_s \rangle \mathd s\right)\right] \right|
    \]
    in terms of Hamilton--Jacobi--Bellman equations $(\psi^{(k)})_{k=0,\dots,m-1}$ on $[t_k,t_{k+1}]$ with forcing:
    \[
        \partial_t \psi^{(k)} + \frac12 \Delta \psi + \frac{1}{2N} |\nabla \psi^{(k)}|^2 = g,\qquad \psi^{(k)}(t_{k+1}) = \psi^{(k+1)}_{t_{k+1}}+\varphi_{k+1},
    \]
    where $\psi^{(m)}=0$.
\end{remark}

By controlling the difference of the moment generating functions of $\langle u_t,\varphi\rangle$ and $\langle \mu_t,\varphi\rangle$ for complex arguments, we also control the difference of the moments. The expectations agree:
\[
    \E[\langle u_t,\varphi\rangle] = \langle p_t \ast \rho_0,\varphi\rangle = \E[\langle \mu_t,\varphi\rangle].
\]
The higher moments are controlled with the help of the Cauchy integral formula.

\begin{corollary}\label{cor:moments}
    Let $m \in \{2,3,\dots\}$ and let $f$, $u_0$, $\gamma$ and $\varphi$ be as in Theorem~\ref{thm:main-nonlocal-testfunction}. Assume that $\kappa^2 m + \kappa^2\rho_{\max} \leq \gamma$. Then the following bound holds uniformly in $t \geq 0$:
    \begin{align*}
        &|\mathbb{E} [\langle u_t, \varphi \rangle^m] -\mathbb{E} [\langle
    \mu_t, \varphi \rangle^m] |   \\
    & \lesssim_m \frac{\|\varphi\|_\infty^{m-1}}{N^{1/2}} \left(\varepsilon^\alpha \|\varphi\|_{C^\alpha}
        + \|\varphi\|_{L^\infty} 
        \exp\left(-c \frac{\rho_{\min}^2}{\kappa^2\rho_{\max}}\right)\right)\\ \nonumber
    & \hspace{10pt}+ \frac{\|\varphi\|_\infty^{m-2}}{N} \left(\varepsilon^{\alpha} \|\nabla \varphi\|_\infty \| \varphi \|_{C^{\alpha}} + (\| \operatorname{id} - f^2 \|_{\infty} + \delta \rho_{\max}^{1/2} \exp\left(-c \frac{\rho_{\min}^2}{\kappa^2\rho_{\max}}\right)) \| \nabla \varphi \|_{\infty}^2\right).
    \end{align*}
\end{corollary}

\begin{proof}
    If $F:\C\to\C$ is an entire function, the generalized Cauchy integral formula gives for any $R > 0$
\[ F^{(n)} (0) = \frac{n!}{2 \pi i} \int_{\partial B_R (0)} \frac{F
   (w)}{w^{n + 1}} \mathd w. \]
We apply this with $F (z) = \exp (z \langle u_t, \varphi \rangle)$ and $G (z)
= \exp (z \langle \mu_t, \varphi \rangle)$. Then 
\[
    \mathbb{E} [\langle u_t, \varphi \rangle^m] =\mathbb{E} [F^{(m)} (0)] =
   \frac{m!}{2 \pi i} \int_{\partial B_R (0)} \frac{\mathbb{E} [\exp (w
   \langle u_t, \varphi \rangle)]}{w^{m + 1}} \mathd w,
\]
and similarly for $\mu_t$, provided that $R>0$ is small enough that the absolute exponential moments exist and are bounded across $|w|=R$, so that Fubini's theorem allows us to exchange expectation and contour integral. This is the case if $\kappa^2\|R\varphi\|_\infty \leq \gamma$, for $\gamma$ as in Theorem~\ref{thm:main-nonlocal-testfunction}, and therefore
\begin{align} \nonumber
  &| \mathbb{E} [\langle u_t, \varphi \rangle^m] -\mathbb{E} [\langle \mu_t,
  \varphi \rangle^m] | = \left| \frac{m!}{2 \pi i} \int_{\partial B_R
  (0)} \frac{\mathbb{E} [\exp (w \langle u_t, \varphi \rangle)] -\mathbb{E}
  [\exp (w \langle \mu_t, \varphi \rangle)]}{w^{m + 1}} \mathd w \right|\\ \nonumber
  & \leq \frac{m!}{2 \pi} \max_{w \in \partial B_R (0)} \left| \frac{|
  \mathbb{E} [\exp (w \langle u_t, \varphi \rangle)] -\mathbb{E} [\exp (w
  \langle \mu_t, \varphi \rangle)] |}{w^{m + 1}} \right| \cdot L (\partial B_R
  (0))\\ \nonumber
  &= \frac{m!}{R^m} \max_{| w | = R} | \mathbb{E} [\exp (w \langle u_t,
  \varphi \rangle)] -\mathbb{E} [\exp (w \langle \mu_t, \varphi \rangle)] | \\ \nonumber
  &\lesssim \frac{m!}{R^{m-1}}\frac{e^{CR\|\varphi\|_\infty}}{N^{1/2}}\left(\varepsilon^\alpha \|\varphi\|_{C^\alpha}
        + \|\varphi\|_{L^\infty} 
        \exp\left(-c \frac{\rho_{\min}^2}{\kappa^2\rho_{\max}}\right)\right)\\ \nonumber
    & \hspace{10pt}+ \frac{m!}{R^{m-2}}\frac{e^{C (R\| \varphi \|_{\infty} + R^2\kappa^2 \rho_{\max}\| \varphi \|_{\infty}^2 )}}{N}\\
    & \hspace{30pt} \times \left(\varepsilon^{\alpha} \|\nabla \varphi\|_\infty \| \varphi \|_{C^{\alpha}} + (\| \operatorname{id} - f^2 \|_{\infty} + \delta \rho_{\max}^{1/2} \exp\left(-c \frac{\rho_{\min}^2}{\kappa^2\rho_{\max}}\right)) \| \nabla \varphi \|_{\infty}^2\right), \label{eq:moment-pr1}
\end{align}
where the last step follows from Theorem~\ref{thm:main-nonlocal-testfunction} with $w\varphi$ in place of $\varphi$, which satisfies $\|w\varphi\| = R \|\varphi\|$ for any norm. We now choose $R = \tfrac{m}{\|\varphi\|_\infty(1+\kappa\sqrt{\rho_{\max}m})}$, which is admissible since we assumed $\kappa^2 m \leq \gamma$. This gives, with changing $C>0$,
\begin{equation}\label{eq:moment-pr2}
    e^{C R\| \varphi \|_{\infty}} + e^{C (R\| \varphi \|_{\infty} + R^2\kappa^2 \rho_{\max}\| \varphi \|_{\infty}^2 )} \lesssim e^{C m},
\end{equation}
and, by Stirling's formula,
\[
    \frac{m!}{R^{m-2}} \lesssim \sqrt{m} \left(\frac{m}{e}\right)^m\left(\frac{\|\varphi\|_\infty(1+\kappa\sqrt{\rho_{\max}m})}{m}\right)^{m-2} \lesssim m^{5/2} e^{-m} (\|\varphi\|_\infty(1+\kappa\sqrt{\rho_{\max}m}))^{m-2}.
\]
By Young's inequality $\kappa \sqrt{\rho_{\max} m} \leq \kappa^2 \rho_{\max} + \kappa^2 m \leq \gamma$,   
and therefore
\begin{equation}\label{eq:moment-pr3}
    \frac{m!}{R^{m-2}} \lesssim_m \|\varphi\|_\infty^{m-2},\qquad \text{and similarly} \qquad \frac{m!}{R^{m-1}} \lesssim_m \|\varphi\|_\infty^{m-1}.
\end{equation}
Combining \eqref{eq:moment-pr1}, \eqref{eq:moment-pr2} and \eqref{eq:moment-pr3}, we obtain the claimed bound
\begin{align*}
  &| \mathbb{E} [\langle u_t, \varphi \rangle^m] -\mathbb{E} [\langle \mu_t,
  \varphi \rangle^m] | \\
    & \lesssim_m \frac{\|\varphi\|_\infty^{m-1}}{N^{1/2}} \left(\varepsilon^\alpha \|\varphi\|_{C^\alpha}
        + \|\varphi\|_{L^\infty} 
        \exp\left(-c \frac{\rho_{\min}^2}{\kappa^2\rho_{\max}}\right)\right)\\ \nonumber
    & + \frac{\|\varphi\|_\infty^{m-2}}{N} \left(\varepsilon^{\alpha} \|\nabla \varphi\|_\infty \| \varphi \|_{C^{\alpha}} + \left(\| \operatorname{id} - f^2 \|_{\infty} + \delta \rho_{\max}^{1/2} \exp\left(-c \frac{\rho_{\min}^2}{\kappa^2\rho_{\max}}\right)\right) \| \nabla \varphi \|_{\infty}^2\right). \qedhere
\end{align*}
\end{proof}

In the previous estimates, we tried to allow for an initial density $\rho_0$ with relatively large concentrations, meaning that we tried to work with the $L^1$ norm of $u_t$ as much as possible, at the price of assuming a uniform control in the test function $\varphi$. In the next result we assume instead that the initial density $\rho_0$ does not have too strong concentration, $\rho_{\max}$ is of moderate size, but the test function might be localized. In other words, we use the $L^\infty$ norm of $u_t$, and $L^1$ or $L^2$ based Besov norms for $\varphi$.

\begin{theorem}\label{thm:main-local-testfunction}
    For $f, \bar u, u, c$ and $\gamma$ as in Theorem~\ref{thm:main-nonlocal-testfunction} let $\varphi \in B^{\alpha}_{2,\infty}(\T^d,\C)$ with $\kappa^2\|\varphi\|_\infty \leq \gamma$ and with $\nabla \varphi \in L^\infty$. Then 
  \begin{align}\nonumber
    &| \mathbb{E} [e^{\langle u_t, \varphi \rangle}] -\mathbb{E} [e^{\langle
    \mu_t, \varphi \rangle}] |  \lesssim_{\alpha} \frac{e^{C\|\varphi\|_{L^1}\rho_{\max}}}{N^{1/2}}\left( \varepsilon^\alpha\|\varphi\|_{B_{2,\infty}^\alpha} \rho_{\max}^{1/2} + \|\varphi\|_{L^4}\rho_{\max}^{1/4}\exp\left(-c \frac{\rho_{\min}^2}{\kappa^2\rho_{\max}}\right)\right)\\ \nonumber
    & \hspace{10pt}+ \frac{e^{C \| \varphi \|_{L^1} \rho_{\max}}}{N}
    \varepsilon^{\alpha} \|\nabla \varphi\|_\infty \| \varphi \|_{B^{\alpha}_{1,\infty}}\rho_{\max} \\
    &\hspace{10pt} + \frac{e^{C \| \varphi \|_{L^1} \rho_{\max}}}{N}(\| \operatorname{id} - f^2 \|_{\infty} + \kappa \rho_{\max}^{1/2} \exp\left(-c \frac{\rho_{\min}^2}{\kappa^2\rho_{\max}}\right)) \| \nabla \varphi \|_{L^2}^2. \label{eq:main-local-testfunction} 
  \end{align}
\end{theorem}

\begin{proof}
Compared to Theorem~\ref{thm:main-nonlocal-testfunction}, we want to exchange $L^\infty$ based norms of the test function $\varphi$ with $L^p$ based norms, for $p$ as small as possible. For the initial condition, we use $p=2$ in Proposition~\ref{prop:IC} and obtain
\begin{align*} 
        \mathbb{E}\left[\left|\langle \mu_0-u_0,\varphi\rangle\right|^2\right]^{1/2}
        \lesssim N^{-1/2}\varepsilon^\alpha\|\varphi\|_{B_{2,\infty}^\alpha} \rho_{\max}^{1/2} + N^{-1/2}\|\varphi\|_{L^4}\rho_{\max}^{1/4}\exp\left(-c \left(\frac{\rho_{\min}^2}{\rho_{\max}\kappa^2}\right)\right),
\end{align*}
which as in Theorem~\ref{thm:main-nonlocal-testfunction} leads to the first term on the right hand side of~\eqref{eq:main-local-testfunction}.

  For the dynamic contribution, we obtain similarly as in the proof of Theorem~\ref{thm:main-nonlocal-testfunction}:
  \begin{align*} \nonumber
    | \mathbb{E} [e^{\langle u_t, \varphi \rangle}] -\mathbb{E} [e^{\langle
    u_0, \varphi_0 \rangle}] | &  \lesssim \frac{1}{N} \int_0^t \mathbb{E} [|e^{2 \langle u_s, \varphi_s
    \rangle}|]^{1 / 2} \| \nabla \varphi \|_{\infty} \varepsilon^{\alpha} \frac{e^{-c(t-s)}}{(t-s)^{1/2}} \|
    \varphi \|_{B^{\alpha}_{1,\infty}} \mathbb{E} [\| u_s \|_{L^\infty}^2]^{1 / 2}
    \mathd s\\ \nonumber
    & \quad + \frac{1}{N} \int_0^t \mathbb{E} [|e^{\langle u_s, \varphi_s
    \rangle}|]    e^{-2c(t-s)}\| \nabla \varphi \|_{L^2}^2 \| \operatorname{id} - f^2 \|_{\infty}
    \mathd s\\ 
    & \quad + \frac{1}{N} \int_0^t \mathbb{E} [|e^{2 \langle u_s, \varphi_s
    \rangle}|]^{1 / 2} e^{-2c(t-s)} \| \nabla \varphi \|_{L^2}^2 \mathbb{E} [\| u_s^-
    \|_{L^\infty}^2]^{1 / 2} \mathd s. 
  \end{align*}
  For the first term we applied Lemma~\ref{lem:K-eps-error} and Lemma~\ref{lem:CH-regularity} to bound
  \begin{align*}
    \| ( \nabla \varphi_s)^2 - ( K_{\varepsilon} \ast \nabla \varphi_s )^2
    \|_{L^1} \lesssim \| \nabla \varphi \|_{\infty} \varepsilon^{\alpha} \frac{e^{-c(t-s)}}{(t-s)^{1/2}} \|
    \varphi \|_{B^{\alpha}_{1,\infty}},
  \end{align*}  
  and for the second and third term we used that by Parseval's identity and Lemma~\ref{lem:CH-regularity}
  \[
   \| K_{\varepsilon} \ast \nabla
    \varphi_s \|_{L^2}^2 \leq \|\nabla \varphi_s\|_{L^2}^2 \lesssim e^{-2c(t-s)} \|\nabla \varphi\|_{L^2}^2 . 
  \]
  For the factor $\E[e^{\lambda \langle u_s,\varphi_s\rangle}]$, $\lambda \in \{1,2\}$, we apply the $L^1$ bound on $\varphi_s$ from Lemma~\ref{lem:CH-regularity} and the bound on the exponential moments of $\|u\|_{L^\infty}$ from Corollary~\ref{cor:exp-moments-u}:
  \begin{equation*}
    \mathbb{E} [|e^{\lambda \langle u_s, \varphi_s \rangle}|] \leq
     \mathbb{E} \left[ e^{\lambda C \| \varphi \|_{L^1} \| u_s \|_{\infty}}
     \right] \lesssim \exp (\lambda C \| \varphi \|_{L^1} \rho_{\max}
     +C\lambda^2 \| \varphi \|_{L^1}^2 \kappa^2 \rho_{\max}) \lesssim e^{C \| \varphi \|_{L^1} \rho_{\max}}.
  \end{equation*}
  For the factor $\E[\| u_s \|_{L^\infty}^2]^{1 / 2}$ we  split $u_s = p_s\ast u_0 + v_s$ and apply Corollary~\ref{cor:vunif}:
  \[
     \sup_{s \geq 0} \mathbb{E} [\| u_s \|_{L^\infty}^2]^{1 / 2} \leq \sup_{s \geq 0} \mathbb{E} [\| p_s\ast u_0 \|_{L^\infty}^2]^{1 / 2} + \sup_{s \geq 0} \mathbb{E} [\| v_s \|_{L^\infty}^2]^{1 / 2} \lesssim   \E[\|u_0\|_{L^\infty} + \delta\|u_0\|_{L^\infty}^{1/2}] \lesssim \rho_{\max}.
  \]
  For the negative part of $u$ we obtain from Corollary~\ref{cor:exp-moments-u}
  \[
    \E[\|u_s^-\|_{L^\infty}^2]^{1/2} \lesssim \kappa \E\left[\|u_0\|_{L^\infty}^{1/2}  \exp \left( -\frac{c}{2} \frac{u_{\min}^2}{\kappa^2 \|u_0\|_\infty}\right)\right] \lesssim \kappa \rho_{\max}^{1/2}\exp \left( -c \frac{\rho_{\min}^2}{\kappa^2 \rho_{\max}}\right),
  \]
  which concludes the proof.
\end{proof}

With the same proof as for Corollary~\ref{cor:moments}, we could also obtain a moment bound for localized test functions.

As an example, we can test $\mu_t$ and $u_t$ against a bump function concentrated at scale $\eta\gg \varepsilon$, which shows that for large $N$ and small $\varepsilon$ the fluctuating field $u_t$ and the particle system $\mu_t$ agree well  down to
a certain spatial resolution and we can recover local (statistical) information about $\mu_t$ from $u_t$.

\begin{example}
    Take $\varphi(x) = \varphi_\eta(x) = \eta^{-d} \varphi_0 (\frac{x-x_0}{\eta})$ for $\varphi_0 \in C^\infty_c(\R^d)$ with $\|\varphi_0\|_{L^1}=1$ and $\eta>0$ and $x_0 \in [0,1]^d$.  Then we have for all $p,q \in [1,\infty]$ and $\alpha>0$
    \[\|\varphi_\eta\|_{B^\alpha_{p,q}} \lesssim \eta^{-\alpha-d(1-\tfrac1p)}  \|\varphi_0\|_{B^\alpha_{p,q}},\qquad \|\varphi_\eta\|_{L^p} \lesssim \eta^{-d(1-\tfrac1p)}  \|\varphi_0\|_{L^p},
    \]
    and therefore
    \begin{align}\nonumber
    &| \mathbb{E} [e^{\langle u_t, \varphi_\eta \rangle}] -\mathbb{E} [e^{\langle
    \mu_t, \varphi_\eta \rangle}] |  \lesssim_{\alpha,\varphi_0} \frac{e^{C\rho_{\max}}}{N^{1/2}}\left( \eta^{-d/2}\left(\frac\varepsilon\eta\right)^\alpha \rho_{\max}^{1/2} + \eta^{-3d/4} \rho_{\max}^{1/4}\exp\left(-c \frac{\rho_{\min}^2}{\kappa^2\rho_{\max}}\right)\right)\\
    & \hspace{10pt}+ \frac{e^{C \rho_{\max}}}{N} \eta^{-d-2} \left(\eta
    \left(\frac{\varepsilon}{\eta}\right)^{\alpha} \rho_{\max} + \| \operatorname{id} - f^2 \|_{\infty} + \kappa \rho_{\max}^{1/2} \exp\left(-c \frac{\rho_{\min}^2}{\kappa^2\rho_{\max}}\right) \right). \nonumber
  \end{align}
\end{example}

\section{Lower bounds on the approximation error for not strictly positive densities} \label{sec:lower_bound} 

In this section we show that the assumption $\rho_{\min} > 0$ is crucial to obtain superpolynomial convergence rates. We give an example with $\rho_{\min} = 0$ for which the SPDE~\eqref{eq:reg_DK} only achieves a limited approximation rate. The reason is that due to the violation of the maximum principle, the SPDE $u$ takes negative values with too high probability, and this limits the approximation quality.

We consider the one-dimensional torus $\mathbb{T}$ and we take $f=\sqrt{\cdot}$ and choose the initial distribution
 \[
 \rho_0 (x)= 2 \mathds{1}_{[\frac{1}{2},1]}(x).
 \]
Note that $\rho_0$ is a bounded probability density, and that $\bar{u}=\rho_0$ because $\rho_{\min}=0$, so that also $u_0 = \rho_0$. Our aim is to obtain a lower bound on the approximation of the second moment. Since our focus is on the dynamic error, we remove the contribution from the initial condition, which in this setting is equivalent to considering the variance instead of the second moment. To simplify one of the computations, we consider the Galerkin projection type multiplier $\chi=\mathds{1}_{[-1,1]}$ here, instead of taking $\chi \in C^\infty_c(\R)$. This is not essential and similar results hold for $\chi \in C^\infty_c(\R)$.

In the example, the approximation rate is limited by the fact that $u_t$ takes negative values with too high probability. The idea is to take a small interval $B_\varepsilon$ of length $O(\varepsilon)$ that is separated from the support $[1/2,1]$ of $u_0$ by distance $O(\varepsilon)$, so that for small times $t\leq \varepsilon^2/|\log \varepsilon|^2$ the convolution with the heat kernel, $p_t \ast u_0$, has essentially no mass on $B_\varepsilon$, while due to the nonlocal convolution with $K_\varepsilon$ the fluctuation $v_t = u_t - p_t \ast u_0$ is comparably large so that it dominates $p_t\ast u_0$ on $B_\varepsilon$, at least for short times. Since the weak error has an accumulating structure (it is a time integral), this propagates to times $t \geq \varepsilon^2/|\log \varepsilon|^2$.

\begin{theorem}
    Assume that $\varepsilon<1, \kappa \leq 1$, and let $\varphi(x) = \sin(2\pi x)$. 
    There exists $c>0$ such that for all $t \geq \frac{\varepsilon^2}{|\log \varepsilon|^2}$
    \begin{align} \nonumber
        & |\E[(\langle u_t,\varphi\rangle - \langle p_t \ast u_0, \varphi\rangle)^2] - \E[(\langle \mu_t, \varphi\rangle - \langle p_t \ast \mu_0, \varphi\rangle)^2]| \\
        &\hspace{40pt} \gtrsim \frac{e^{-4\pi^2 t}}{N}  \frac{\varepsilon^{5/2}}{|\log \varepsilon|^{3}} (N^{-1/2} - \exp(-c|\log\varepsilon|^2)).
    \end{align}
\end{theorem}
\begin{remark}
    In particular, since $\exp(-c|\log \varepsilon|^2)$ decays faster than any monomial, we obtain
\[
     |\E[(\langle u_t,\varphi\rangle - \langle p_t \ast u_0, \varphi\rangle)^2] - \E[(\langle \mu_t, \varphi\rangle - \langle p_t \ast \mu_0, \varphi\rangle)^2]| \gtrsim \frac{e^{-4\pi^2 t}}{N^{3/2}}  \frac{\varepsilon^{5/2}}{|\log \varepsilon|^{3}} 
\]
as $N\to \infty$ whenever $\varepsilon = \varepsilon(N) \gtrsim N^{-a}$ for some $a>0$. This does not quite match the upper bound $O(N^{-3/2} \varepsilon^{-d/2})$ that Theorem~\ref{thm:main-nonlocal-testfunction} gives in the case $\rho_{\min} = 0$, but it shows that the weak error does not decay superpolynomially unless $\rho_{\min}>0$.

In this example it seems as if taking $\varepsilon$ relatively large, so that $(N^{-1/2} - \exp(-c|\log\varepsilon|^2))$ becomes negative, we might achieve a higher approximation rate. But that is due to the test function $\varphi = \sin(2\pi x)$ not having any Fourier modes with $|k|> 1$ and because we are considering only the second moment. For general test functions, in particular those of limited smoothness, or for higher moments or the Laplace transform, we need $\varepsilon>0$ sufficiently small to get a good approximation.
\end{remark}

\begin{proof}
    The processes $M^{\mu,t}_s = \langle \mu_s, p_{t-s}\ast \varphi\rangle - \langle \mu_0, p_t \ast \varphi\rangle$ and $M^{u,t}_s = \langle u_s, p_{t-s}\ast \varphi\rangle - \langle u_0, p_t \ast \varphi\rangle$, $s \in [0,t]$, are continuous martingales starting from $0$ and with quadratic variation
    \[
        \langle M^{\mu,t}\rangle_s = \frac1N \int_0^s \langle \mu_r, |\partial_x p_{t-r}\ast \varphi|^2\rangle \mathd r,\qquad \langle M^{u,t}\rangle_s = \frac1N \int_0^s \langle u^+_r, |K_\varepsilon\ast \partial_x p_{t-r}\ast \varphi|^2\rangle \mathd r,
    \]
    respectively. Therefore,
    \begin{align}\nonumber
        & N|\E[(\langle u_t,\varphi\rangle - \langle p_t \ast u_0, \varphi\rangle)^2] - \E[(\langle \mu_t, \varphi\rangle - \langle p_t \ast \mu_0, \varphi\rangle)^2]| \\
        &\hspace{40pt} = N|\E[(M^{u,t}_t)^2] - \E[(M^{\mu,t}_t)^2]| \nonumber \\
        &\hspace{40pt} = \left| \int_0^t \E[\langle u^+_r, |K_\varepsilon\ast \partial_x p_{t-r}\ast \varphi|^2\rangle] \mathd r - \int_0^t \E[\langle \mu_r, |\partial_x p_{t-r}\ast \varphi|^2\rangle] \mathd r\right| \nonumber \\
        & \hspace{40pt} = \int_0^t \E[\langle u^-_r, |\partial_x p_{t-r}\ast \varphi|^2\rangle] \mathd r, \label{eq:lower-bound-pr1}
    \end{align}
    where we used that $K_\varepsilon\ast \partial_x p_{t-r} \ast \varphi =  \partial_x p_{t-r} \ast \varphi$ because $\varphi$  has vanishing $k$-th Fourier modes for all $|k|>1$ and that $\chi(\varepsilon \cdot 1) = \chi(\varepsilon \cdot(-1))= 1$, and we also used that $u_r^+ = u_r + u_r^-$ and $\E[\langle u_r, f\rangle] = \E[\langle \mu_r,f\rangle] = \langle p_r\ast \rho_0, f\rangle$. We continue to estimate the right hand side as follows:
    \begin{align} \nonumber
        \int_0^t \E[\langle u^-_r, |\partial_x p_{t-r}\ast \varphi|^2\rangle] \mathd r & \geq \int_0^t \int_{B_\varepsilon} \E[u^-_r(x)] e^{-(2\pi)^2(t-r)} (2\pi)^2 |\cos(2\pi x)|^2 \mathd x \mathd r \\
        & \gtrsim e^{-4\pi^2 t}\int_0^t \int_{B_\varepsilon} \E[u^-_r(x)]  \mathd x \mathd r, \label{eq:lower-bound-pr2}
    \end{align}
    where $B_\varepsilon = [\tfrac{\varepsilon}{4\pi}, \tfrac{3\varepsilon}{4\pi}]$ and we used that $|\cos(2\pi x)|^2 \geq |\cos(3/2)|^2 >0$ for $x \in B_\varepsilon$. Now  Lemma~\ref{lem:lower-bound-Beps-int} gives for     $r \leq \frac{\varepsilon^2}{|\log \varepsilon|^2}$:
    \begin{equation}\label{eq:lower-bound-pr3}
        \int_{B_\varepsilon} \mathbb{E}[u_r^{-}(x)] \mathd x \gtrsim \varepsilon^{-1/2}  \sqrt{r}(N^{-1/2} - \exp(-c|\log \varepsilon|^2)).
    \end{equation} 
    Now \eqref{eq:lower-bound-pr1}, \eqref{eq:lower-bound-pr2} and \eqref{eq:lower-bound-pr3} together yield
    \begin{align*}
        & |\E[(\langle u_t,\varphi\rangle - \langle p_t \ast u_0, \varphi\rangle)^2] - \E[(\langle \mu_t, \varphi\rangle - \langle p_t \ast \mu_0, \varphi\rangle)^2]|\\
        & \hspace{30pt} \gtrsim \frac{e^{-4\pi^2 t}}{N}  \varepsilon^{-1/2}  \left(t \wedge \frac{\varepsilon^2}{|\log \varepsilon|^2}\right)^{3/2} (N^{-1/2} - \exp(-c|\log\varepsilon|^2))\\
        & \hspace{30pt} \gtrsim \frac{e^{-4\pi^2 t}}{N}  \frac{\varepsilon^{5/2}}{|\log \varepsilon|^{3}} (N^{-1/2} - \exp(-c|\log\varepsilon|^2)),
    \end{align*}
    assuming $t \geq \frac{\varepsilon^2}{|\log \varepsilon|^2}$ in the last step.   
\end{proof}

In order to derive the lower bound for $\E[u_t^-]$, recall that the mild solution $u$ to \eqref{eq:mild_formulation} satisfies $u_t = p_t \ast u_0 + v_t$. We start by deriving a lower bound on $\mathbb{E}[| v_t(x)|]$, for $x \in B_\varepsilon = [\frac{\varepsilon}{4 \pi}, \frac{3 \varepsilon}{4 \pi}]$. Writing
\begin{equation}
q^\varepsilon (t,x) := \nabla(K_\varepsilon \ast p_t)(x) = \nabla K_\varepsilon \ast p_t(x),
\end{equation}
we have
\begin{equation}
    v_t(x) =  \frac{1}{\sqrt{N}} \int_{[0, t] \times \mathbb{T}^d} q^\varepsilon(t-s,x-y)\cdot f(u_s^+(y)) \mathd W(s,y).
\end{equation}

\begin{lemma}
    Let $\tau>0$ and $x \in B_\varepsilon:= [\frac{\varepsilon}{4 \pi}, \frac{3 \varepsilon}{4 \pi}]$. There exists $C>0$ such that whenever $t \leq \tau (\frac{\varepsilon}{2 \pi})^2$, we have 
    \begin{equation}\label{eq:v-lower-bound}
         \mathbb{E}[|v_t(x)|]\gtrsim N^{-1/2} e^{-3\tau} \varepsilon^{-3/2}\sqrt{t}.
    \end{equation}
\end{lemma}
\begin{proof}
Write $v_t(x)=M_t^{t,x}$, where $(M^{t,x}_r)_{r \in [0,t]}$ is a continuous martingale with quadratic variation
\begin{equation}
  \langle M^{t,x} \rangle_r = \frac{1}{N}\int_0^r \int_\mathbb{T}  |q^\varepsilon(t-s,x-y)|^2 f(u_s^+(y))^2 \mathd y \mathd s.
\end{equation}
We bound $\mathbb{E}[|v_t(x)|]$ from below with H\"older's inequality,
$\mathbb{E}[Z^2]\le \mathbb{E}[|Z|]^{2/3}\,\mathbb{E}[Z^4]^{1/3}$ for all $Z \in L^4(\Omega)$, and therefore
\begin{equation}\label{eq:PZ}
    \mathbb{E}[|v_t(x)|]\geq \frac{\mathbb{E}[v_t(x)^2]^{3/2}}{\mathbb{E}[v_t(x)^4]^{1/2}}.
\end{equation}
Lower bound on the second moment: By It\^o's isometry, we have
\[
    \mathbb{E}[v_t(x)^2]=\mathbb{E}[\langle M^{t,x}\rangle_t]
    =\frac1N\int_0^t\int_\mathbb{T} |q^\varepsilon(t-s,x-y)|^2\,
    \mathbb{E}[u_s^+(y)]\,\mathd y\,\mathd s.
\]
Since $u_s^+\geq u_s$ pointwise and $\mathbb{E}[v_s]=0$, we have
$\mathbb{E}[u_s^+(y)]\geq \mathbb{E}[u_s(y)]=(p_s\ast u_0)(y)\geq 1$ for
$y\in[1/2,1]$. 
By Lemma~\ref{lem:q-estimate}, for
$t\leq \tau(\tfrac{\varepsilon}{2\pi})^2$ the bound
$|q^\varepsilon(t-s,x-y)|^2\geq C e^{-2\tau}\varepsilon^{-4}$ holds for
$y\in A_x:=\{y\in[1/2,1]:1+x-y\in[\tfrac{\varepsilon}{2\pi},\tfrac{\varepsilon}{\pi}]\}$,
and we have $|A_x|\geq\tfrac{\varepsilon}{4\pi}$. Hence
\begin{equation}\label{eq:2nd-moment}
    \mathbb{E}[v_t(x)^2]\geq \frac1N\int_0^t\int_{A_x}
    |q^\varepsilon(t-s,x-y)|^2\,\mathd y\,\mathd s
    \gtrsim \frac1N e^{-2\tau}\varepsilon^{-4}\,|A_x|\,t
    \gtrsim N^{-1}e^{-2\tau}\varepsilon^{-3}t.
\end{equation}
Upper bound on the fourth moment: By the Burkholder--Davis--Gundy inequality and then the Cauchy-Schwarz inequality
\begin{align*}
    \mathbb{E}[v_t(x)^4] & \lesssim \mathbb{E}\big[\langle M^{t,x}\rangle_t^2\big] \leq \frac{\int_0^t \|q^\varepsilon(t-s)\|_{L^2}^2 \mathd s}{N^2} \int_0^t\E[\|u_s^+\|_\infty^2] \|q^\varepsilon(t-s)\|_{L^2}^2 \mathd s.
\end{align*}
Using $\sup_{s\geq0}\mathbb{E}[\|u_s\|_\infty^2]\lesssim 1$, which follows from
Corollary~\ref{cor:vunif} and $\|p_s\ast u_0\|_\infty\leq\|u_0\|_\infty=2$,
together with $\|q^\varepsilon(r)\|_{L^2}\leq\|\partial_x K_\varepsilon\|_{L^2}
\lesssim\varepsilon^{-3/2}$, we obtain
\begin{equation}\label{eq:4th-moment}
    \mathbb{E}[v_t(x)^4]\lesssim N^{-2}t^2\varepsilon^{-6}.
\end{equation}
Combining \eqref{eq:PZ}, \eqref{eq:2nd-moment} and
\eqref{eq:4th-moment},
\[
    \mathbb{E}[|v_t(x)|]\geq
    \frac{\mathbb{E}[v_t(x)^2]^{3/2}}{\mathbb{E}[v_t(x)^4]^{1/2}}
    \gtrsim \frac{(N^{-1}e^{-2\tau}\varepsilon^{-3}t)^{3/2}}
    {(N^{-2}t^2\varepsilon^{-6})^{1/2}}
    = N^{-1/2}e^{-3\tau}\varepsilon^{-3/2}\sqrt{t}. \qedhere
\]
\end{proof}


\begin{lemma} \label{lem:lower-bound-Beps-int} 
Let $B_\varepsilon:= [\frac{\varepsilon}{4 \pi}, \frac{3 \varepsilon}{4 \pi}]$. There exists a constant $c >0$ such that if $0 < t \leq \frac{\varepsilon^2}{|\log \varepsilon|^2}$
\begin{equation}
\mathbb{E}\left[\int_{B_\varepsilon} u_t^{-}(x) \mathd x\right] \gtrsim \varepsilon^{-1/2}  \sqrt{t}(N^{-1/2} - \exp(-c|\log \varepsilon|^2)).
\end{equation} 
\end{lemma}

\begin{proof}
    Note that $p_t \ast u_0 \geq 0$ so that, since $(a+z)^{-} \geq z^- - a$ for $a \geq 0$,
\begin{align*}
\mathbb{E}\left[\int_{B_\varepsilon} u_t^-(x) \mathd x \right] 
    &\geq \int_{B_\varepsilon}  \mathbb{E}[v_t(x)^-]\mathd x - \int_{B_\varepsilon}  (p_t \ast u_0)(x)\mathd x  \\
    &= \frac{1}{2}\int_{B_\varepsilon}  \mathbb{E}[|v_t(x)|]\mathd x - \int_{B_\varepsilon}  (p_t \ast u_0)(x)\mathd x,
\end{align*}
where in the second step we used that for each fixed $x$, $v_t(x)$ is centered because it is given by a stochastic integral, which implies $\mathbb{E}[v_t(x)^-] = \frac{1}{2} \mathbb{E}[|v_t(x)|]$.

Using the lower bound~\eqref{eq:v-lower-bound} for $\E[|v_t(x)|]$, for $x \in B_\varepsilon$ and $t \leq \tau (\varepsilon/2\pi)^2$,  we get
\[
\frac{1}{2}\int_{B_\varepsilon}  \mathbb{E}[|v_t(x)|]\mathd x \gtrsim |B_\varepsilon| N^{-1/2} e^{-3\tau} \varepsilon^{-3/2} \sqrt{t} \gtrsim  N^{-1/2} e^{-3\tau} \varepsilon^{-1/2} \sqrt{t},
\]
where we have used that $|B_\varepsilon| = \varepsilon/(2 \pi)$.

It remains to estimate the deterministic contribution. The following heat kernel estimate holds on $\mathbb{T}$ and locally uniformly in time:
\begin{equation}\label{heat_estimate}
    p_t(x) \leq C t^{-1/2} \exp \left(-\frac{d_{\mathbb{T}}(x,0)^2}{2t}\right),
\end{equation}
which follows from the explicit expression $p_t(x) = \sum_{k \in \Z} (2\pi t)^{-1/2} \exp\left(-\frac{(x-k)^2}{2t}\right)$.
Since $(p_t \ast u_0)(x) = 2 \int_{\frac12}^1p_t(x-y)\mathd y$, and  the support $[1/2,1]$ of $u_0$ is at torus distance
$\gtrsim \varepsilon$ from $B_\varepsilon$, we deduce that
\begin{equation*}
    (p_t \ast u_0) (x) \lesssim \frac{1}{\sqrt{t}}\exp \left( - C\frac{\varepsilon^2}{2t}\right), \qquad x \in B_\varepsilon = \left[\frac{\varepsilon}{4 \pi}, \frac{3 \varepsilon}{4 \pi}\right].
\end{equation*}
Combining this with the stochastic estimate we obtain
\begin{align*}
\mathbb{E}\left[\int_{B_\varepsilon}u_t^-(x)\mathd x\right] &\geq \frac{1}{2}\int_{B_\varepsilon} \mathbb{E}[|v_t(x)|]dx - \int_{B_\varepsilon} (p_t \ast u_0)(x) dx \\
    &\gtrsim |B_\varepsilon| N^{-1/2} e^{-3\tau}\varepsilon^{-3/2}\sqrt{t}    - |B_\varepsilon| t^{-1/2} \exp(-C \frac{\varepsilon^2}{t}) \\
    &\gtrsim N^{-1/2} e^{-3\tau} \varepsilon^{-1/2} \sqrt{t}   -(\varepsilon^{-1/2} \sqrt{t}) (\frac{\varepsilon}{t} \exp(-C \frac{\varepsilon^2}{t})).
\end{align*}
For 
$ 
t \leq \frac{\varepsilon^2}{|\log \varepsilon|^2}
$ 
we obtain the claimed bound, because then we can take $\tau \lesssim 1/(|\log \varepsilon|^2) \lesssim 1$, and we have $(\frac{\varepsilon}{t} \exp(-C \frac{\varepsilon^2}{t})) \leq \varepsilon \exp(-\tfrac{C}{2}\frac{\varepsilon^2}{t}) \leq \varepsilon \exp(-c|\log\varepsilon|^2)$.
\end{proof}

\section{Upper bounds for the negative part of $u$}\label{sec:negative}

In this section, we consider a general probability density $u_0 \in L^\infty(\T^d)$ as initial condition for $u$, and we treat $u_0$ as deterministic.
If $u_0$ is random, all the expectations and probabilities should be understood conditionally on $u_0$.

Recall that $v_t = u_t - p_t \ast u_0$. Since 
$p_t \ast u_0$ preserves pointwise the lower bound of the initial condition, 
any negative values of $u_t$ must come from $v_t$. This lets us control the negative part of $u_t$ through the size of $v_t$, which is captured by the following lemma. 

\begin{lemma}
  \label{lem:negative-u-moments}Let $u_{\min} := \min_x u_0 (x) \geq 0$. Then, for every $r > 0$ and $p \in [1,\infty]$, 
the following estimate holds
   \[ \mathbb{E} [\| u^-_t \|_{L^p}^r] \leq \mathbb{E} [\| v_t \|_{L^p}^{2
     r}]^{1 / 2} \mathbb{P} (\| v_t \|_{\infty} > u_{\min})^{1 / 2} . \]
\end{lemma}

\begin{proof}
    Since the mapping $x \mapsto x^-:=\max\{-x,0\}$ is monotone decreasing, 
  and since the positivity of $p_t$ implies the uniform lower bound $p_t \ast u_0 (x)\geq u_{\min}$ for all $(t,x) \in \mathbb R_+ \times \mathbb{T}^d$,
  we have 
  \begin{equation}\label{estimate:neg_u}
       u_t^- =(p_t \ast u_0 + v_t)^- \leq  (u_{\min} + v_t)^- \leq v_t^- \mathds{1}_{\{\|v_t\|_\infty > u_{\min}\}},
  \end{equation}
  where in the last step we used that $u_{\min}\ge 0$ and $(\cdot)^-$ is decreasing, and that $(u_{\min} + v_t)^- \equiv 0$ for $\| v_t \|_{\infty} \leq u_{\min}$, since $u_{\min} + v_t \geq u_{\min} - \| v_t \|_{\infty}$.
Now the claim follows by taking the $L^p$ norm on both sides of \eqref{estimate:neg_u}, raising both sides to the $r$-th power, taking the expectation and then applying the Cauchy-Schwarz inequality to upper bound the right hand side.
\end{proof}

Thus, to control the moments of $u^-$, it remains to estimate the moments of $\|v_t\|_{L^p}$ and the tail probability $\mathbb{P} (\| v_t \|_{\infty} > u_{\min}) $. Recall that by Assumption~\ref{ass:kappa} we have $\delta \le \kappa \leq 1$.

\begin{lemma}
  \label{lem:v-Lebesgue-moments}
For every $p \in [2, \infty)$, we have
\[ 
\sup_{t \geq 0}\mathbb{E} \left[\| v_t \|_{L^p}^p \right]^{1 / p} \lesssim \sqrt{p}\delta \|u_0\|^{1/2}_{L^{p/2}}+p\delta^2\lesssim_p \delta \| u_0 \|_{L^{p /
     2}}^{1 / 2} ,
     \]
where the implicit constant in the first inequality is independent of $p$.
\end{lemma}

\begin{proof}
  Although the stochastic  convolution is  not a martingale in $t$, 
  applying the Burkholder--Davis--Gundy inequality  to
  the martingale 
  $$ r \mapsto\frac{1}{\sqrt{N}} \int_{[0, r] \times \mathbb{T}^d}
  q^{\varepsilon} (t - s, x - y) f (u_s^+(y)) \cdot \mathd W (s, y)
  $$
  on $[0,t]$, i.e. freezing the $t$ in $q^{\varepsilon} (t - \cdot)$, gives
  \begin{align*}
    \sqrt{N} \mathbb{E} [| v_t (x) |^p]^{1 / p} & \simeq \sqrt{p} \mathbb{E}
    \left[ \left( \int_0^t \int_{\mathbb{T}^d} | q^{\varepsilon} (t - s, x -
    y) f (u_s^+(y)) |^2 \mathd y \mathd s \right)^{p / 2} \right]^{1 / p}\\
    & \lesssim \sqrt{p} \mathbb{E} \left[ \left( \int_0^t
    \int_{\mathbb{T}^d} | q^{\varepsilon} (t - s, x - y) |^2 (| p_s \ast u_0
    (y) | + | v_s (y) |) \mathd y \mathd s \right)^{p / 2} \right]^{1 / p}\\
    & \leq \sqrt{p} \left( \int_0^t \int_{\mathbb{T}^d} |
    q^{\varepsilon} (t - s, x - y) |^2 (| p_s \ast u_0 (y) | + \| v_s (y)
    \|_{L^{p / 2} (\Omega)}) \mathd y \mathd s \right)^{1 / 2}.
  \end{align*}
  Here we used that $f (u^+)^2 \lesssim u^+ \leq | u |$, and we applied
  Minkowski's inequality, where we need $p \geq 2$, to pull the $\|
  \cdot \|_{L^{p / 2} (\Omega)}$ norm inside the integration. 
  
  We raise both
  sides to the $p$-th power and integrate in $x$, to obtain
  \begin{align*}
    & N^{p / 2} \mathbb{E} [\| v_t \|_{L^p}^p] \\
    & \lesssim p^{p / 2} \int_{\mathbb{T}^d} \left( \int_0^t
    \int_{\mathbb{T}^d} \frac{| q^{\varepsilon} (t - s, x - y) |^2}{\| |
    q^{\varepsilon} |^2 \|_{L^1 ([0, t] \times \mathbb{T}^d)}} (| p_s \ast u_0
    (y) | + \| v_s (y) \|_{L^{p / 2} (\Omega)}) \| | q^{\varepsilon} |^2
    \|_{L^1 ([0, t] \times \mathbb{T}^d)} \mathd y \mathd s \right)^{p / 2}
    \mathd x\\
    & \leq p^{p / 2} \int_{\mathbb{T}^d} \left( \int_0^t
    \int_{\mathbb{T}^d} \frac{| q^{\varepsilon} (t - s, x - y) |^2}{\| |
    q^{\varepsilon} |^2 \|_{L^1 ([0, t] \times \mathbb{T}^d)}} (| p_s \ast u_0
    (y) | + \| v_s (y) \|_{L^{p / 2} (\Omega)})^{p / 2} \| | q^{\varepsilon}
    |^2 \|_{L^1 ([0, t] \times \mathbb{T}^d)}^{p / 2} \mathd y \mathd s
    \right) \mathd x .
 \end{align*}
In the last step we used Jensen's inequality with respect to the normalised measure proportional to $| q^{\varepsilon} (t - s, x - y) |^2\mathd y \mathd s$. Hence, 
    \begin{align*}
     N^{p / 2} \mathbb{E} [\| v_t \|_{L^p}^p]
     & \lesssim p^{p / 2} \| | q^{\varepsilon} |^2 \|_{L^1 ([0, t] \times
    \mathbb{T}^d)}^{p / 2} \sup_{s \leq t} \left( \int_{\mathbb{T}^d} (|
    p_s \ast u_0 (y) | + \| v_s (y) \|_{L^{p / 2} (\Omega)})^{p / 2} \mathd y
    \right)\\
    & \leq p^{p / 2} \| q^{\varepsilon} \|_{L^2 ([0, t] \times
    \mathbb{T}^d)}^p (\| u_0 \|_{L^{p / 2} (\mathbb{T}^d)}^{1 / 2} + \sup_{s
    \leq t} \| v_s \|_{L^{p / 2} (\Omega \times \mathbb{T}^d)}^{1 / 2})^p.
  \end{align*} 
  Now write $M_t^{(p)} \assign \sup_{s \leq t} \| v_s \|_{L^p (\Omega \times
     \mathbb{T}^d)}$. Since $\Omega \times \mathbb{T}^d$  is a probability space, $M_t^{(p / 2)} \leq M_t^{(p)}$. Therefore,
  \begin{align*}
    M^{(p)}_t & \leq C \frac{\sqrt{p}}{\sqrt{N}} \| q^{\varepsilon}
    \|_{L^2 ([0, t] \times \mathbb{T}^d)} \left( \| u_0 \|_{L^{p / 2}
    (\mathbb{T}^d)}^{1 / 2} + \sqrt{M^{(p)}_t} \right)\\
    & \leq \frac{C \sqrt{p}}{\sqrt{N}} \| q^{\varepsilon} \|_{L^2 ([0,
    t] \times \mathbb{T}^d)} \| u_0 \|_{L^{p / 2} (\mathbb{T}^d)}^{1 / 2} +
    C^2 \frac{p}{2 N} \| q^{\varepsilon} \|_{L^2 ([0, t] \times
    \mathbb{T}^d)}^2 + \frac{1}{2} M^{(p)}_t,
  \end{align*}
  and then
  \[ 
  M^{(p)}_t \lesssim \frac{\sqrt{p}}{\sqrt{N}} \| q^{\varepsilon} \|_{L^2
     ([0, t] \times \mathbb{T}^d)} \| u_0 \|_{L^{p / 2} (\mathbb{T}^d)}^{1 /
     2} + \frac{p}{N} \| q^{\varepsilon} \|_{L^2 ([0, t] \times
     \mathbb{T}^d)}^2 . 
  \]
  Since $\delta = N^{- 1 / 2} \varepsilon^{- d / 2}$, 
  it remains to show $\| q^{\varepsilon} \|_{L^2 ([0, t] \times \mathbb{T}^d)}
  \lesssim \varepsilon^{- d / 2}$.  Using an idea from \cite{Cornalba2023}, we observe that $(t,x)\mapsto p_t \ast K_{\varepsilon}(x)$ solves the heat
  equation with initial condition $K_{\varepsilon}$, so by energy estimates
  for the heat equation:
  \[ \| q^{\varepsilon} \|_{L^2 ([0, t] \times \mathbb{T}^d)} = \| \nabla p
     \ast K_{\varepsilon} \|_{L^2 ([0, t] \times \mathbb{T}^d)} \lesssim \|
     K_{\varepsilon} \|_{L^2 (\mathbb{T}^d)} \lesssim \varepsilon^{- d /
     2}, \]
   uniformly in $t$.
   This gives
   \[
   \sup_{t \geq 0}\mathbb{E}[\|v_t\|^p_{L^p}]^{1/p} \lesssim\sqrt{p}\delta \|u_0\|^{1/2}_{L^{p/2}}+p\delta^2,
   \]
   and since $\delta \leq 1$ and $u_0$ is a probability density on the unit torus and thus $\|u_0\|^{1/2}_{L^{p/2}} \gtrsim 1$
   \[
   \sqrt{p}\delta \|u_0\|^{1/2}_{L^{p/2}}+p\delta^2 \lesssim_p \delta \|u_0\|^{1/2}_{L^{p/2}},
   \]
   which completes the proof. 
\end{proof}

\begin{corollary}\label{cor:vunif}
  For all $p \in [2,\infty)$, the uniform norm of $v$ is controlled uniformly in time by
  \begin{equation}
    \sup_{t \geq 0} \mathbb{E} [\| v_t \|^p_{\infty}]^{1 / p} \lesssim \left(\sqrt{p}  \kappa \| u_0 \|_{\infty}^{1 /
    2} + p \kappa^2 \right) \lesssim_p \kappa \| u_0
     \|_{\infty}^{1 / 2}, \label{eq:moment-bound-uniform}
  \end{equation}
  where the implicit constant in the first inequality is independent of $p$.
\end{corollary}

\begin{proof}
  Since $v_t$ is a convolution with $K_{\varepsilon}$, and since $K_{\varepsilon}$ is localized at frequencies $\lesssim \varepsilon^{-1}$, 
  Bernstein's inequality \cite[Lemma 2.1]{Bahouri2011} gives
  \begin{align*}
        \mathbb{E} [\| v_t \|_{\infty}^p]^{1 / p} &\lesssim \varepsilon^{- d / p}
     \mathbb{E} [\| v_t \|_{L^p}^p]^{1 / p} \lesssim e^{\tfrac{d}{p} \log(1/\varepsilon)} \left(
     \sqrt{p} \delta \| u_0 \|_{L^{p / 2}}^{1 / 2} + p \delta^2 \right) \\
     &\leq
     e^{\tfrac{d}{p} \log(1/\varepsilon)} \left( \sqrt{p} \delta \| u_0 \|_{\infty}^{1 / 2} + p
     \delta^2 \right).
    \end{align*}
  If $p \geq p_0 = d \log \varepsilon^{- 1}$, then the prefactor is
  bounded by $e$. 
  Hence in that case
  \begin{align*}
         \mathbb{E} [\| v_t \|_{\infty}^p]^{1 / p} &\lesssim
  \sqrt{p} \delta \| u_0 \|_{\infty}^{1 / 2} + p
     \delta^2 \leq \sqrt{p} \kappa \|u_0\|_\infty^{1/2} + p \kappa^2.
    \end{align*}
  For $p < p_0$, we use the monotonicity of moments and bound
  \begin{align*}
        \mathbb{E} [\| v_t \|_{\infty}^p]^{1 / p} &\leq \mathbb{E} [\| v_t
     \|_{\infty}^{p_0}]^{1 / (p_0)}\lesssim \left( \sqrt{p_0} \delta \|
     u_0 \|_{\infty}^{1 / 2} + p_0 \delta^2 \right) \leq \sqrt{p} \kappa \|u_0\|_\infty^{1/2} + p \kappa^2.
     \end{align*}
For the second inequality in \eqref{eq:moment-bound-uniform}, we use $\kappa \leq 1$ and $\|u_0\|_\infty \geq 1$, which imply $ \kappa^2 \leq \kappa \leq  \|u_0\|_\infty^{1/2}$. 
\end{proof}

From the previous result together with Lemma~\ref{lem:tail-estimate} in the Appendix, we deduce the following corollary.
\begin{corollary}
  \label{cor:v-uniform-tails} 
  There exist constants $c,C>0$, independent of $t,N, \varepsilon$ and $u_0$, such that for all $x\geq 0$

 \begin{equation}
      \mathbb{P} (\| v_t \|_{\infty} \geq x) \leq C \exp \left[ -c \left( \frac{x^2}{\kappa^2\|u_0\|_\infty } \wedge \frac{x}{\kappa^2}\right)\right],
 \end{equation}
 and for all $\lambda \in
  [0, c / \kappa^2]$
  \begin{equation} \label{eq:exp-moments-v}
       \mathbb{E} [e^{\lambda \| v_t \|_{\infty}}] \leq C \exp ( C \lambda^2 \kappa^2 \|u_0\|_\infty). 
  \end{equation}
\end{corollary}

\begin{proof}
Set $X:=\| v_t\|_\infty$. Denoting the implicit constant from the moment bound  \eqref{eq:moment-bound-uniform} by $K$, we obtain
\[
\mathbb{E}[X^p]^{1/p} \leq \beta \sqrt{p} + \gamma p, \qquad p \geq 2, 
\]
where
\[ 
   \beta := K \kappa \| u_0 \|_{\infty}^{1 / 2},
     \qquad \gamma := K \kappa^2 . 
\]
Then applying Lemma~\ref{lem:tail-estimate} to $X$, and substituting the values of $\beta$ and $\gamma$,  we obtain the claimed estimates.
\end{proof}

\begin{corollary}
  \label{cor:exp-moments-u}
  For $p,r \in [1,\infty)$, the negative part of $u$ is bounded by
  \begin{equation}\label{eq:moments-negative-u}
      \sup_{t \geq 0} \E[\|u_t^-\|_{L^p}^r]^{1/r} \lesssim_r \delta \|u_0\|_{L^{r\vee (p/2)}}^{1/2}  \exp \left( -\frac{c}{r} \frac{u_{\min}^2}{\kappa^2 \|u_0\|_\infty}\right),
  \end{equation}
  and also
  \begin{equation}\label{eq:maximum-negative-u}
      \sup_{t \geq 0} \E[\|u_t^-\|_{L^\infty}^r]^{1/r} \lesssim_r \kappa \|u_0\|_{L^\infty}^{1/2}  \exp \left( -\frac{c}{r} \frac{u_{\min}^2}{\kappa^2 \|u_0\|_\infty}\right).
  \end{equation}
  Moreover, there exist  $c, C > 0$ such that   for all $\lambda \in
  [0, c/ \kappa^2]$ and all $p \in [1,\infty]$
  \begin{equation}\label{eq:exp-moments-u} \mathbb{E} [e^{\lambda \| u_t \|_{L^p}}] \leq C \exp (\lambda\|u_0\|_{L^p} + C \lambda^2 \kappa^2\|u_0\|_\infty ).
  \end{equation}
\end{corollary}

\begin{proof}
For the negative part we apply Lemma~\ref{lem:negative-u-moments} together with Lemma~\ref{lem:v-Lebesgue-moments} and Corollary~\ref{cor:v-uniform-tails}: If $2r \geq p$, we bound
\begin{align*}
    \mathbb{E} [\| u^-_t \|_{L^p}^r]^{1/r} & \leq \mathbb{E} [\| v_t \|_{L^p}^{2 r}]^{1 / 2r} \mathbb{P} (\| v_t \|_{\infty} > u_{\min})^{1 / 2r} \\
    & \leq \E[\| v_t \|_{L^{2r}}^{2 r}]^{1 / 2r} \mathbb{P} \left(\| v_t \|_{\infty} > u_{\min}\right)^{1 / 2r} \\
    & \lesssim_r \delta \|u_0\|_{L^r}^{1/2}  \exp \left[ -\frac{c}{r}\left( \frac{u_{\min}^2}{\kappa^2 \|u_0\|_\infty} \wedge \frac{u_{\min}}{\kappa^2}\right)\right] \\
    & \lesssim \delta \|u_0\|_{L^r}^{1/2}  \exp \left( -\frac{c}{r} \frac{u_{\min}^2}{\kappa^2\|u_0\|_\infty}\right),
\end{align*}
and if $2r < p$ we bound instead
\[
    \mathbb{E} [\| v_t \|_{L^p}^{2 r}]^{1 / 2r} \leq \mathbb{E} [\| v_t \|_{L^p}^{p}]^{1 / p} \lesssim \delta \|u_0\|_{L^{p/2}}^{1/2}.
\]
Similarly, we obtain for the uniform norm by Corollary~\ref{cor:vunif}:
\begin{align*}
    \mathbb{E} [\| u^-_t \|_{L^\infty}^r]^{1/r} \leq \mathbb{E} [\| v_t \|_{L^\infty}^{2 r}]^{1 / 2r} \mathbb{P} (\| v_t \|_{\infty} > u_{\min})^{1 / 2r} \lesssim_r \kappa \|u_0\|_{L^\infty}^{1/2}  \exp \left( -\frac{c}{r} \frac{u_{\min}^2}{\kappa^2\|u_0\|_\infty}\right).
\end{align*}
To bound the exponential moments, we decompose $u_t = p_t \ast u_0 + v_t$, hence
\[
\mathbb{E}\left[\exp \left(\lambda \| u_t \|_{L^p} \right) \right] \leq \exp (\lambda\|p_t \ast u_0\|_{L^p}) \mathbb{E}\left[\exp\left(\lambda \| v_t \|_\infty \right)\right]
\]
and then we use that $\|p_t \ast u_0\|_{L^p} \leq \|u_0\|_{L^p}$ and the bound~\eqref{eq:exp-moments-v} from the previous corollary for the second factor. 
\end{proof}

\section{Initial approximation of particles by a density}\label{sec:initial_approx}

Recall that
\[
    \delta = \varepsilon^{-d/2}N^{-1/2},\qquad \kappa = \delta \sqrt{1\vee d\log \frac{1}{\varepsilon}} .
\]

\begin{proposition}\label{prop:IC}
Let $d \geq 1$, let $X_0^1,\ldots,X_0^N$ be i.i.d.
 with common density $\rho_0$ on $\mathbb{T}^d$, and assume that $\rho_0(x) \in [\rho_{\min},\rho_{\max}]$ for all $x \in \mathbb T^d$. Let $\mu_0 := \frac1N \sum_{i=1}^N \delta_{X_0^i}$ and let $K_\varepsilon$ be as in Section~\ref{sec:main}, i.e. $K_\varepsilon = \mathcal{F}^{-1} (\chi(\varepsilon\cdot))$ for a symmetric function $\chi \in C^\infty_c(\R^d)$ with $\chi \equiv 1$ on a neighborhood of $0$. Define
\[
        \nu := K_\varepsilon *(\mu_0-\rho_0).
\]
We set
\[
        \bar{u} :=
        \begin{cases}
        \rho_0 + \nu,
        & \|\nu\|_\infty \leq \frac{\rho_{\min}}{2},\\
        \rho_0,
        & \text{otherwise}.
        \end{cases}
\]
Then, for every $\alpha>0$ and every $\varphi \in C^\alpha(\mathbb{T}^d,\C)$, 
\[
        \mathbb{E}\left[\left|\langle \mu_0-\bar{u},\varphi\rangle\right|^2\right]^{1/2}
        \lesssim N^{-1/2} \min_{p \in \{2,\infty\}} \left\{\varepsilon^\alpha \rho_{\max}^{1/p}\|\varphi\|_{B^\alpha_{p,\infty}}
        + \rho_{\max}^{1/(2p)}\|\varphi\|_{L^{2p}}
        \exp\left(-c \frac{\rho_{\min}^2}{\kappa^2\rho_{\max}}\right) \right\},
\]
where $\rho_{\max}^{1/\infty} := 1$ and the constants depend only on $d$, $\alpha$, and $\chi$.
\end{proposition}

\begin{proof}
By arguing componentwise, we may assume that $\varphi$ is real-valued. Let
\[
        A :=\left\{ \|\nu\|_{L^\infty} \le \frac{\rho_{\min}}{2}\right\}.
\]
On $A$, we have \(\bar{u}=\rho_0+\nu\), and hence, using the symmetry of \(K_\varepsilon\),
\[
        \langle \mu_0-\bar{u},\varphi\rangle
        =
        \langle \mu_0-\rho_0,\varphi\rangle
        -
        \langle K_\varepsilon*(\mu_0-\rho_0),\varphi\rangle
        =
        \langle \mu_0-\rho_0,(1-K_\varepsilon)\varphi\rangle,
\]
where we slightly abuse notation by writing $(1-K_\varepsilon)\varphi:=\varphi-K_\varepsilon\ast \varphi$. On $A^c$, we have \(\bar{u}=\rho_0\), and therefore
\[
        \langle \mu_0-\bar{u},\varphi\rangle
        =
        \langle \mu_0-\rho_0,\varphi\rangle .
\]
Thus
\begin{equation}\label{eq:initial-pr1}
        \mathbb{E}\left[\left|\langle \mu_0-\bar{u},\varphi\rangle\right|^2\right]^{1/2}
        \leq \mathbb{E}\left[\left|\langle \mu_0-\rho_0,(1-K_\varepsilon)\varphi\rangle\right|^2\right]^{1/2} +
        \mathbb{E}\left[\left|\langle \mu_0-\rho_0,\varphi\rangle\right|^2\mathds{1}_{A^c}\right]^{1/2}.
\end{equation}
We estimate the two terms separately. The random variables $(\psi(X^i_0) - \mathbb E[\psi(X^i_0)])_{i=1,\dots, N}$ are centered and independent and thus martingale increments. So, for any $p\geq 2$ the Burkholder--Davis--Gundy inequality and Minkowski's inequality yield
\begin{align*}
    \mathbb E[|\langle \mu_0 - \bar{u},\psi\rangle|^p]^{1/p} & \lesssim_p \frac1N \mathbb E\left[\left( \sum_{i=1}^N |\psi(X^i_0) - \mathbb E[\psi(X^i_0)]|^2\right)^{p/2}\right]^{1/p}\\
    & \leq \frac1N \mathbb E\left[\sum_{i=1}^N \||\psi(X^i_0) - \mathbb E[\psi(X^i_0)]|^2\|_{L^{p/2}(\Omega)}\right]^{1/2} \\
    & \lesssim N^{-1/2} \min\left\{\|\psi\|_{L^\infty(\mathbb T^d)}, \left(\int_{\T^d}|\psi(x)|^p\rho_0(x) \mathd x\right)^{1/p}\right\} \\
    & \leq N^{-1/2} \rho_{\max}^{1/p} \|\psi\|_{L^\infty(\mathbb T^d)},
\end{align*}
where we interpret $\rho_{\max}^{1/\infty}:=1$. Together with an application of the Cauchy-Schwarz inequality, this reduces~\eqref{eq:initial-pr1} to
\begin{align} \nonumber
            &\mathbb{E}\left[\left|\langle \mu_0-\bar{u},\varphi\rangle\right|^2\right]^{1/2}\\ \nonumber
            & \lesssim \min_{p \in \{2,\infty\}} \Bigl\{ N^{-1/2}\rho_{\max}^{1/p} \|(1-K_\varepsilon)\varphi\|_{L^p(\mathbb T^d)} + N^{-1/2}\rho_{\max}^{1/(2p)}\|\varphi\|_{L^{2p}(\mathbb T^d)}\mathbb P(A^c)^{1/4}\Bigr\} \\
        & \lesssim \min_{p \in \{2,\infty\}} \Bigl\{ N^{-1/2} \rho_{\max}^{1/p} \varepsilon^\alpha\|\varphi\|_{B^\alpha_{p,\infty}} + N^{-1/2} \rho_{\max}^{1/(2p)}\|\varphi\|_{L^{2p}(\mathbb T^d)} \mathbb P(A^c)^{1/4} \Bigr\}, \label{eq:initial-pr2}
\end{align}
using Lemma~\ref{lem:K-eps-error} in the second step. To conclude the proof, we have to estimate $\mathbb P(A^c)$. We start by estimating $\mathbb P(|\nu(x)| \ge r)$ for a fixed $x \in \mathbb T^d$ and $r \ge 0$ by Bernstein's inequality. For that purpose, note that $(K_\varepsilon(x-X^i_0) - \mathbb E[K_\varepsilon(x-X^i_0)])_{i=1,\dots,N}$ are centered i.i.d. with variance
\begin{align*}
    \operatorname{Var}(K_\varepsilon(x-X^i_0) - \mathbb E[K_\varepsilon(x-X^i_0)]) & \leq \int_{\mathbb T^d} |K_\varepsilon(x-y)|^2 \rho_0(y)\mathd y \leq \|K_\varepsilon\|_{L^2}^2 \rho_{\max} \\
    & \lesssim \varepsilon^{-d} \|K_\varepsilon\|_{L^1} \rho_{\max} \simeq \varepsilon^{-d} \rho_{\max}
\end{align*}
which follows from the Bernstein inequality for frequency localized functions in \cite[Lemma 2.1]{Bahouri2011}, and this inequality also yields
\[
    |K_\varepsilon(x-X^i_0) - \mathbb E[K_\varepsilon(x-X^i_0)]| \lesssim \|K_\varepsilon\|_{L^\infty} \lesssim \varepsilon^{-d} \|K_\varepsilon\|_{L^1} \simeq \varepsilon^{-d}.
\]
Therefore, Bernstein's inequality for independent centered random variables gives for $Y_i = K_\varepsilon(x-X^i_0) - \mathbb E[K_\varepsilon(x-X^i_0)]$ and for some $c>0$
\begin{align*}
    \mathbb P(|\nu(x)| \ge r) & = \mathbb P\left(\left|\frac1N \sum_{i=1}^N Y_i\right| \ge r\right) \\
    & \leq \exp\left( -\frac{\tfrac12 N^2 r^2}{N\E[Y_i^2] + \tfrac13 \|Y_i\|_{L^\infty(\Omega)} Nr} \right)\\
    & \leq \exp\left( -\frac{\tfrac12 N^2 r^2}{NC\varepsilon^{-d} \rho_{\max} + \tfrac13 C\varepsilon^{-d}Nr} \right) \\
    & \leq \exp\left( -c \left(\frac{r^2}{(\varepsilon^{-d/2} N^{-1/2}\rho_{\max}^{1/2})^2}\wedge \frac{r}{\varepsilon^{-d}N^{-1}} \right)\right)\\
    & = \exp\left( -c \left(\frac{r^2}{(\delta\rho_{\max}^{1/2})^2}\wedge \frac{r}{\delta^2} \right)\right).
\end{align*}
To pass from this estimate in a single point to a uniform bound, we first estimate the moments of $\|\nu\|_{L^p}$. The second part of Lemma~\ref{lem:tail-estimate} gives for some $C>0$
\begin{align*}
    \mathbb E[\|\nu\|_{L^p}^p] & = \int_{\mathbb T^d} \mathbb E[|\nu(x)|^p]\mathd x \leq C^p (\delta \rho_{\max}^{1/2})^{p} p^{p/2} + C^p (\delta^2)^p p^p,  
\end{align*}
and therefore,
\[
    \mathbb E[\|\nu\|_{L^p}^p]^{1/p} \lesssim \delta \sqrt{\rho_{\max}} \sqrt{p} + \delta^2 p.
\]
The next steps are very similar to the arguments in Section~\ref{sec:negative}: We apply again \cite[Lemma 2.1]{Bahouri2011}, using that $\nu$ is frequency localized at frequencies $\lesssim 1/\varepsilon$, and we obtain
\[
    \mathbb E[\|\nu\|_{L^\infty}^p]^{1/p} \lesssim \varepsilon^{-d/p} \mathbb E[\|\nu\|_{L^p}^p]^{1/p} \lesssim \varepsilon^{-d/p}\left(\delta\sqrt{\rho_{\max}}  \sqrt{p} + \delta^2 p\right).
\]
For $p\geq d \log \tfrac1\varepsilon$ the factor $\varepsilon^{-d/p}$ is of order $1$. For smaller $p$ we use the monotonicity of moments and obtain with $p_0 = d \log \tfrac1\varepsilon$:
\[
    \mathbb E[\|\nu\|_{L^\infty}^p]^{1/p} \leq \mathbb E[\|\nu\|_{L^\infty}^{p_0}]^{1/p_0} \lesssim \delta \sqrt{\rho_{\max}}  \sqrt{p_0} + \delta^2 p_0 \lesssim \sqrt{\rho_{\max}}\kappa + \kappa^2,
\]
so that we can estimate for all $p \geq 1$:
\[
    \mathbb E[\|\nu\|_{L^\infty}^p]^{1/p} \lesssim \sqrt{\rho_{\max}} \kappa \sqrt{p} + \kappa^2 p.
\]
Now Lemma~\ref{lem:tail-estimate} shows that
\[
    \mathbb P(A^c) = \mathbb P\left(\|\nu\|_{L^\infty} \geq \frac{\rho_{\min}}{2}\right) \lesssim \exp\left(-c \left(\frac{\rho_{\min}^2}{\rho_{\max}\kappa^2}\wedge \frac{\rho_{\min}}{\kappa^2}\right)\right) = \exp\left(-c \left(\frac{\rho_{\min}^2}{\rho_{\max}\kappa^2}\right)\right),
\]
which together with~\eqref{eq:initial-pr2} concludes the proof.
\end{proof}

\section{Duality and the Hamilton--Jacobi--Bellman equation}\label{sec:HJB}

Let $\mu_t = \tfrac1N \sum_{k=1}^N \delta_{B^k_t}$ be the empirical measure of $N$ independent Brownian motions on $\mathbb T^d$. Here we discuss that the duality of $\mu$ with the Hamilton--Jacobi--Bellman equation from \cite{Konarovskyi2019} extends to complex valued test functions, and we derive a Cole--Hopf representation and regularity estimates for the Hamilton--Jacobi--Bellman equation. 

\begin{lemma}\label{lem:duality}
    Let $\varphi\in C^{1,2}(\R_+\times \T^d,\C)$. Then
    \[
        \mathd e^{\langle \mu_t, \varphi_t \rangle} = e^{\langle \mu_t, \varphi_t \rangle}  \left\langle \mu_t,  \partial_t \varphi_t + \frac{1}{2} \Delta \varphi_t + \frac{1}{2 N} (\nabla  \varphi_t)^2 \right\rangle \mathd t  + \mathd M_t,
    \]
    where $(\nabla \varphi_t)^2 = \nabla \varphi_t\cdot \nabla \varphi_t$ is not to be confused with $|\nabla \varphi_t|^2$, and $M$ is a complex-valued martingale.
\end{lemma}

\begin{proof}
    Applying It\^o's formula to real and imaginary part, we obtain
\[ \mathd \varphi_t (B^k_t) = \partial_t \varphi_t(B^k_t)\mathd t + \frac{1}{2} \Delta \varphi_t (B^k_t)\mathd t + \nabla
   \varphi_t (B^k_t) \cdot \mathd B^k_t. \]
In the following we write
    \[
        [M]=[M,M] = [\mathrm{Re}(M), \mathrm{Re}(M)] - [\mathrm{Im}(M),\mathrm{Im}(M)] + 2i[\mathrm{Re}(M),\mathrm{Im}(M)]
    \]
    for the usual quadratic variation. This gives
\[ \mathd \langle \mu_t, \varphi_t \rangle = \left\langle \mu_t, \partial_t \varphi_t + \frac{1}{2}
   \Delta \varphi_t \right\rangle \mathd t + \mathd M_t (\varphi), \qquad \mathd
   [ M (\varphi)]_t = \frac{1}{N} \langle \mu_t, (\nabla
   \varphi_t)^2 \rangle \mathd t. \]
The exponential function is holomorphic, and therefore It\^o's formula for holomorphic functions of complex-valued semimartingales yields
\begin{align*}
  \mathd e^{\langle \mu_t, \varphi_t \rangle} & = e^{\langle \mu_t,
  \varphi_t \rangle} \left( \langle \mu_t, \partial_t \varphi_t \rangle \mathd
  t + \left\langle \mu_t, \frac{1}{2} \Delta \varphi_t \right\rangle \mathd t
  + \mathd M_t (\varphi) + \frac{1}{2} \mathd [ M (\varphi)
  ]_t \right)\\
  & = e^{\langle \mu_t, \varphi_t \rangle} \left( \left\langle \mu_t,
  \partial_t \varphi_t + \frac{1}{2} \Delta \varphi_t + \frac{1}{2 N} (\nabla
  \varphi_t)^2 \right\rangle \mathd t + \mathd M_t (\varphi) \right) .
\end{align*}
Since $\varphi \in C^{1,2}(\R_+\times \T^d,\C)$, the functions $e^{\langle \mu_t, \varphi_t\rangle}$ and $\nabla \varphi$ are bounded locally in time, and therefore the stochastic integral on the right hand side is a true martingale.
\end{proof}

Recall that $(p_t)_{t\geq 0}$ is the periodic heat kernel for $\tfrac12 \Delta$. Next, we derive the Cole--Hopf formula for the complex-valued Hamilton--Jacobi--Bellman equation, which is slightly more subtle than in the real-valued case because there could be cancellations and we need to guarantee that $p_t \ast e^{\tfrac1N\varphi}$ stays bounded away from zero. Here we consider the equation forward in time, so that in the application in Section~\ref{sec:main} we have to reverse time.

\begin{lemma}\label{lem:Cole-Hopf}
    Let $N \in \N$ and let $\varphi \in C^2(\T^d,\C)$ be such that $\|\tfrac1N \varphi\|_\infty \leq \tfrac\pi4$. Let $\log$ be the principal branch of the logarithm. Then the function
    \begin{equation}\label{eq:Cole-Hopf}
       \varphi_t = N \log \left( p_t \ast e^{\frac{1}{N} \varphi} \right) ,\qquad t \geq 0,
    \end{equation}
    is a classical solution to the Hamilton--Jacobi--Bellman equation
    \begin{equation}\label{eq:HJB}
       \partial_t \varphi_t - \frac{1}{2} \Delta \varphi_t - \frac{1}{2 N} (\nabla
   \varphi_t)^2 = 0,\qquad \varphi_0 = \varphi,
    \end{equation}
    and there exists $C>0$, depending only on $\tfrac\pi4$, such that $|e^{\frac1N \varphi_t}|\in \left[C^{-1},C\right]$ for all $t \geq 0$.
\end{lemma}

\begin{proof}
    The definition/regularity of $\varphi_t$ has issues if $p_{t} \ast e^{\frac{1}{N} \varphi}$ takes values in the negative half-line, where the logarithm is not holomorphic. The condition $\| \frac{1}{N} \varphi\|_{\infty} \leq \frac{\pi}{4}$ rules this out, because then $| \operatorname{Arg} ( e^{\frac{1}{N} \varphi} )| \leq\frac{\pi}{4}$, so $|\operatorname{Im} ( e^{\frac{1}{N} \varphi} ) | \leq
   | \operatorname{Re} ( e^{\frac{1}{N} \varphi} ) |$ and thus
\begin{equation*}
  \operatorname{Re} \left( e^{\frac{1}{N} \varphi} \right)   \geq \sqrt{\frac{1}{2} \operatorname{Re} \left( e^{\frac{1}{N} \varphi}
  \right)^2 + \frac{1}{2} \operatorname{Im} \left( e^{\frac{1}{N} \varphi}
  \right)^2} = \sqrt{\frac{1}{2}} \left| e^{\frac{1}{N} \varphi} \right| = \sqrt{\frac{1}{2}} e^{ \operatorname{Re}(\frac{1}{N}\varphi)} \geq \sqrt{\frac{1}{2}} e^{- \frac{\pi}{4}} .
\end{equation*}
Therefore,
\[ \operatorname{Re} \left( p_t \ast e^{\frac{1}{N} \varphi} \right) = p_t \ast
   \operatorname{Re} \left( e^{\frac{1}{N} \varphi} \right) \geq
   \sqrt{\frac{1}{2}} e^{- \frac{\pi}{4}}, \]
and of course also
\[ \left| p_t \ast e^{\frac{1}{N} \varphi} \right| \geq \left| \operatorname{Re}
   \left( p_t \ast e^{\frac{1}{N} \varphi} \right) \right| = \left| p_t \ast
   \operatorname{Re} \left( e^{\frac{1}{N} \varphi} \right) \right| \geq
   \sqrt{\frac{1}{2}} e^{- \frac{\pi}{4}}, \]
and
\[ \left| p_t \ast e^{\frac{1}{N} \varphi} \right| \leq e^{\frac{\pi}{4}}
   . \]
Since $e^{\tfrac1N \varphi_t} = p_t \ast e^{\frac{1}{N} \varphi}$, we have shown the claim $|e^{\tfrac1N \varphi_t}| \in [C^{-1},C]$ for suitable $C>0$. We have also shown that there exists
an open rectangle $R = (- C, C) i + (c, C)$ with $c,C>0$ such that
\begin{equation}\label{eq:support-heat-convolution}
    p_t \ast e^{\frac{1}{N} \varphi}  \in R.
\end{equation}
Since $R$ is bounded away from the negative half-axis, the logarithm is holomorphic
and bounded on $R$, and therefore the following computation is justified:
\[
    \partial_t \varphi_t = N \frac{\partial_t p_t \ast e^{\frac{1}{N} \varphi} }{p_t \ast e^{\frac{1}{N} \varphi}} = N \frac{\frac12 \Delta p_t \ast e^{\frac{1}{N} \varphi} }{p_t \ast e^{\frac{1}{N} \varphi}}
\]
and
\[
    \frac12 \Delta \varphi_t = \frac12 N \frac{\Delta p_t \ast e^{\frac1N \varphi}}{p_t \ast e^{\frac{1}{N} \varphi}} - \frac12 N \frac{(\nabla p_t \ast e^{\frac1N \varphi})^2}{(p_t \ast e^{\frac1N \varphi})^2} = \frac12 N \frac{\Delta p_t \ast e^{\frac1N \varphi}}{p_t \ast e^{\frac{1}{N} \varphi}} - \frac{1}{2N} (\nabla \varphi_t)^2,
\]
which concludes the proof. 
\end{proof}

Next, we derive global regularity estimates for the Cole--Hopf solution to the Hamilton--Jacobi--Bellman equation.

\begin{lemma}\label{lem:CH-regularity}
    Let $N \in \N$, let $\alpha> 0$, let $p,q \in [1,\infty]$, and let $\varphi \in B^\alpha_{p,q}$ be such that $\|\tfrac1N\varphi\|_\infty \leq \tfrac{\pi}{4}$. Then there exists $c>0$ such that the following estimates hold uniformly in $t\geq 0$:
    \begin{equation}\label{eq:CH-Besov}
        \sup_{t \geq 0} \left\|N \log p_t \ast e^{\tfrac1N\varphi}\right\|_{B^\alpha_{p,q}} + \sup_{t \geq 0} \sqrt{t} e^{ct} \left\|N \nabla \log p_t \ast e^{\tfrac1N\varphi}\right\|_{B^{\alpha}_{p,q}} \lesssim_{\alpha,p,q} \|\varphi\|_{B^\alpha_{p,q}},
    \end{equation}
    as well as
    \begin{equation}\label{eq:CH-Lebesgue}
         \sup_{t \geq 0} \left\|N \log p_t \ast e^{\tfrac1N\varphi}\right\|_{L^p} + \sup_{t \geq 0} \sqrt{t} e^{ct} \left\| N \nabla \log p_t \ast e^{\tfrac1N\varphi}\right\|_{L^p}\lesssim_p \|\varphi\|_{L^p},
    \end{equation}
    and
    \begin{equation}\label{eq:CH-Lebesgue-gradient}
         \sup_{t \geq 0} e^{ct} \left\| N \nabla \log p_t \ast e^{\tfrac1N\varphi}\right\|_{L^p} \lesssim_p \|\nabla \varphi\|_{L^p}.
    \end{equation}
\end{lemma}

\begin{proof}
    By~\cite[Chapter 5.5, Theorem~2]{Runst1996}, if $\Phi$ is a smooth function with $\Phi(0) = 0$, then for any $\beta \geq 1$ and $p,q \in [1,\infty]$ and $f \in B^\beta_{p,q}(\R^d)$:
    \begin{equation}\label{eq:composition}
        \|\Phi(f)\|_{B^\beta_{p,q}(\R^d)} \lesssim_{\beta,\|f\|_\infty,\Phi} \|f\|_{B^\beta_{p,q}(\R^d)}.
    \end{equation}
    This is formulated for functions from $\R^m$ to $\R$, and extends to functions from $\C$ to $\C$ by identification of $\C$ with $\R^2$ and by treating the image componentwise. The same proof works with minor modifications on $B^\beta_{p,q}(\T^d)$, using the increment characterization of $B^\beta_{p,q}(\T^d)$ from \cite[Chapter 3.5.4]{Schmeisser1987}. Alternatively, the result on $\T^d$ can be deduced from that on $\R^d$ by periodization and restriction, as discussed in \cite[Section 5]{Ehrnstrom2018}. The same result is true for $\beta \in (0,1)$, which can be shown directly by the increment characterization of $B^\beta_{p,q}(\T^d)$ from \cite[Chapter 3.5.4]{Schmeisser1987}.
    
    We apply~\eqref{eq:composition} first with $\log(\cdot + m) - \log( m)$ 
    and with $m:=\int \exp(\tfrac1N \varphi) \mathd x$. This function is holomorphic on $\C\setminus\{-m+x: x \in (-\infty, 0]\}$.  By~\eqref{eq:support-heat-convolution} the convolution $p_t \ast \exp(\tfrac1N \varphi)$ takes values  in $R =  (- C, C) i + (c, C)$ for $c,C>0$, so $p_t \ast \exp(\tfrac1N \varphi) - m$ takes values in $R-m$, which is bounded away from $\{-m+x: x \in (-\infty, 0]\}$. 

    Therefore, we can find a smooth function $\Phi$ such that
    \[
        \Phi(z) = \log(z + m) - \log( m),\qquad z \in R-m,
    \]
    for example by multiplication with a smooth cutoff. Then~\eqref{eq:composition} gives for any $\beta>0$
    \[
        \left\|N \Phi \left(p_t \ast \left(e^{\tfrac1N\varphi} - m\right)\right) \right\|_{B^\beta_{p,q}} \lesssim_{\beta,p,q} N \left\|p_t \ast \left(e^{\tfrac1N\varphi} - m\right) \right\|_{B^\beta_{p,q}}.
    \]
    Now we choose a dyadic partition of unity for the Littlewood-Paley blocks which satisfies $\Delta_{-1}f = \int f \mathd x$ for all $f \in \mathcal S'(\T^d)$. The Besov norms for any two dyadic partitions of unity are equivalent, see \cite[Chapter~3.5, Theorem~1]{Schmeisser1987} so this change costs only a constant depending on $\beta,p,q$. Then
    \begin{align*}
        \left\|p_t \ast \left(e^{\tfrac1N\varphi} - m\right)\right\|_{B^\beta_{p,q}} & = \left\| p_t \ast \left(e^{\tfrac1N\varphi} - 1\right) - (m-1)\right\|_{B^\beta_{p,q}} \\
        & \simeq_{\beta,p,q} \left\| \left(2^{j\beta} \|\Delta_j p_t \ast (e^{\tfrac1N\varphi} - 1)\|_{L^p} \right)_{j \geq 0} \right\|_{\ell^q_j} \\
        & \lesssim_p \left\| \left(2^{j\beta} e^{-c t 2^{2j}} \|\Delta_j (e^{\tfrac1N\varphi} - 1)\|_{L^p} \right)_{j \geq 0} \right\|_{\ell^q_j} \\
        & \lesssim_{\beta,p,q} e^{-c t} \min\bigl\{ \|e^{\tfrac1N\varphi}  - 1\|_{B^{\beta}_{p,q}}, t^{-1/2}\|e^{\tfrac1N\varphi}  - 1\|_{B^{\beta-1}_{p,q}} \bigr\},
    \end{align*}
    for some $c>0$, where we applied~\cite[Lemma~2.4]{Bahouri2011} in the third line, which by Poisson summation extends from $\R^d$ to $\T^d$, see Chapter~3 of~\cite{Gubinelli2015EBP} for similar arguments, and in the last step we used that
    \[
        e^{-c t 2^{2j}} \lesssim \min\{e^{-ct}, t^{-1/2} 2^{-j} e^{-ct}\},
    \]
    with changing $c>0$. Now it remains to apply~\eqref{eq:composition} once more, this time with $\Phi(z) = e^z-1$, which gives the first bound if $\beta>0$ and both bounds if $\beta>1$:
    \begin{align}\nonumber
        N \left\| \log p_t \ast e^{\tfrac1N\varphi} - \log m\right\|_{B^\beta_{p,q}} & \lesssim N e^{-c t} \min\left\{\|\frac1N\varphi\|_{B^\beta_{p,q}}, t^{-1/2} \|\frac1N\varphi\|_{B^{\beta-1}_{p,q}}\right\} \\
        & = e^{-ct} \min\left\{\|\varphi\|_{B^\beta_{p,q}}, t^{-1/2} \|\varphi\|_{B^{\beta-1}_{p,q}}\right\}. \label{eq:CH-regularity-pr1}
    \end{align}
    The bound for the second term in~\eqref{eq:CH-Besov} now follows from
    \[
        N \left\| \nabla \log p_t \ast e^{\tfrac1N\varphi} \right\|_{B^\alpha_{p,q}} \lesssim N \left\| \log p_t \ast e^{\tfrac1N\varphi} - \log m\right\|_{B^{\alpha+1}_{p,q}},
    \]
    and then we apply~\eqref{eq:CH-regularity-pr1} with $\beta=\alpha+1$. For the first estimate in~\eqref{eq:CH-Besov}, 
    it now suffices to estimate
    \[
        N |\log m| = N \left| \int_0^1 \partial_s \left(\log \int e^{s\tfrac1N \varphi} \mathd x \right)\mathd s\right| \leq N  \int_0^1 \left| \frac{\int e^{s\tfrac1N \varphi} \frac1N \varphi \mathd x}{\int e^{s\tfrac1N  \varphi} \mathd x}\right| \mathd s \lesssim \|\varphi\|_{L^1} \lesssim \|\varphi\|_{B^\alpha_{p,q}},
    \]
    where the last step used that $\alpha>0$. The estimate~\eqref{eq:CH-Lebesgue} follows by a similar differentiation and integration in an artificial parameter $s$. The bound~\eqref{eq:CH-Lebesgue-gradient} follows by the chain rule for the gradient, and another application of the spectral gap estimate for $\Delta$ on $\T^d$, that is $\|p_t \ast \Delta_j f\|_{L^p} \lesssim e^{-ct2^{2j}} \|\Delta_j f \|_{L^p}$ for $j \geq 0$.
\end{proof}

\section{Numerical experiments}\label{sec:numerics}

For the numerical experiments, we rescale the diffusion and noise in \eqref{eq:reg_DK} and consider
\begin{equation}
\mathrm{d}u_t
=
D\Delta u_t\,\mathrm{d}t
+
\frac{1}{\sqrt{N}}
\nabla\cdot
\left(
K_\varepsilon*
\left(
\sqrt{2D u_t^+}\,\mathrm{d}W_t
\right)
\right),
\label{eq:reg_DK_diffusivity}
\end{equation}
which amounts to rescaling time and considering \eqref{eq:reg_DK} on the time scale $2 D t$.

Throughout this section, we work on the one-dimensional torus $\mathbb{T}$ and
consider the mollifying kernel $K_\varepsilon$ with $\mathcal F K_\varepsilon = \mathds{1}_{[-\varepsilon^{-1},\varepsilon^{-1}]}$.
 This kernel violates the assumptions of Section~\ref{sec:main} because the indicator function is not smooth, but at the price of a logarithmic loss in the error estimates we could extend the analysis to this case, see Lemma~8.7 of \cite{Gubinelli2017KPZ}. Assuming that $\varepsilon = L^{-1}$ with $L \in \N$, we therefore restrict the simulation to the spectral Galerkin space
\[
V_L
=
\operatorname{span}
\left\{
e^{2\pi ikx}:|k|\leq L
\right\},
\]
and we approximate the solution $u$ at discrete times $t_n=n\Delta t$ by
\[
\rho_L^n(x)
=
\sum_{|k|\leq L}
\widehat{\rho}_k^n e^{2\pi ikx}.
\]

As the heat semigroup is a diagonal operator in Fourier space, its action
over one time step is given by
\[
e^{D\Delta t\,\partial_{xx}}e^{2\pi ikx}
=
e^{-D(2\pi k)^2\Delta t}e^{2\pi ikx}.
\]
The nonlinear stochastic term is computed pseudospectrally on an
equidistant physical-space grid
\(
x_j=\frac{j}{M},
\text{ for } j=0,\ldots,M-1,
\)
where $M\geq 2L+1$. Firstly, $\rho_L^n$ is evaluated at grid points from its Fourier coefficients. 
Then, at each grid point, the stochastic flux increment is computed as follows
\[
F_j^n
=
\frac{1}{\sqrt{N}}
\sqrt{2D\max\{\rho_L^n(x_j),0\}}
\sqrt{M\Delta t}\,\xi_j^n,
\qquad
\xi_j^n\sim\mathcal{N}(0,1),
\]
where \(\xi_j^n\) are independently drawn across grid points and time
steps.
Lastly, the stochastic flux is projected into the spectral Galerkin space by 
\[
\widehat{F}_k^n
=
\frac{1}{M}
\sum_{j=0}^{M-1}
F_j^n e^{-2\pi ikj/M}, \qquad |k|\leq L.
\]

The derivative then simply corresponds to multiplication
with $2\pi ik$. 
Combining this with the heat-semigroup
multiplier gives the update
\[
\widehat{\rho}_k^{\,n+1}
=
e^{-D(2\pi k)^2\Delta t}
\left(
\widehat{\rho}_k^n
+
2\pi ik\,\widehat{F}_k^n
\right),
\qquad |k|\leq L.
\]
Since both the diffusion and the conservative stochastic contribution
vanish for $k=0$, the zero Fourier coefficient remains unchanged and the
total mass is preserved.

\subsection{Sample trajectories}\label{sect:sample_traj}

We first illustrate the qualitative behavior of the spectral Galerkin
approximation by considering individual trajectories with several particle
numbers and spectral resolutions. The initial probability density is chosen as the smooth positive profile
\begin{equation}\label{eq:positive_molifier}
\rho_0(x)
=
\rho_{\min}
+
(\rho_{\max}-\rho_{\min})
\psi\left(\frac{d_{\mathbb T}(x,x_c)}{r}\right),
\end{equation}
where $0<\rho_{\min}<1<\rho_{\max}$, where $\psi(x) = C \exp(-(1-x^2)^{-1})$ on $(-1,1)$ and $0$ outside, with $C$ chosen so that the integral of $\psi$ equals $1$, and where $x_c = 0.7$. The value of $r$ is chosen so that $\int_\T \rho_0(x) \mathd x = 1$.

For the numerical initialization, $\rho_0$ is sampled at the uniform grid
points $x_j=j/(2L+1)$ and 
renormalized to ensure unit mass.

For each particle number $N$, the
expected counts $N\Delta x\,\rho_0(x_j)$ are converted into deterministic
integer counts by taking their floors and assigning the remaining particles
to the entries with the largest fractional remainders. The resulting
histogram,
\[
\rho_{h,j}^0=\frac{n_j}{N\Delta x},
\qquad
\sum_j n_j=N,
\]
therefore has exact discrete unit mass, while its minimum and maximum differ
slightly from the prescribed values $\rho_{\min}$ and $\rho_{\max}$.

\begin{figure}[!htbp]
    \centering
    \captionsetup{font=footnotesize}
    \includegraphics[
        width=\linewidth
    ]{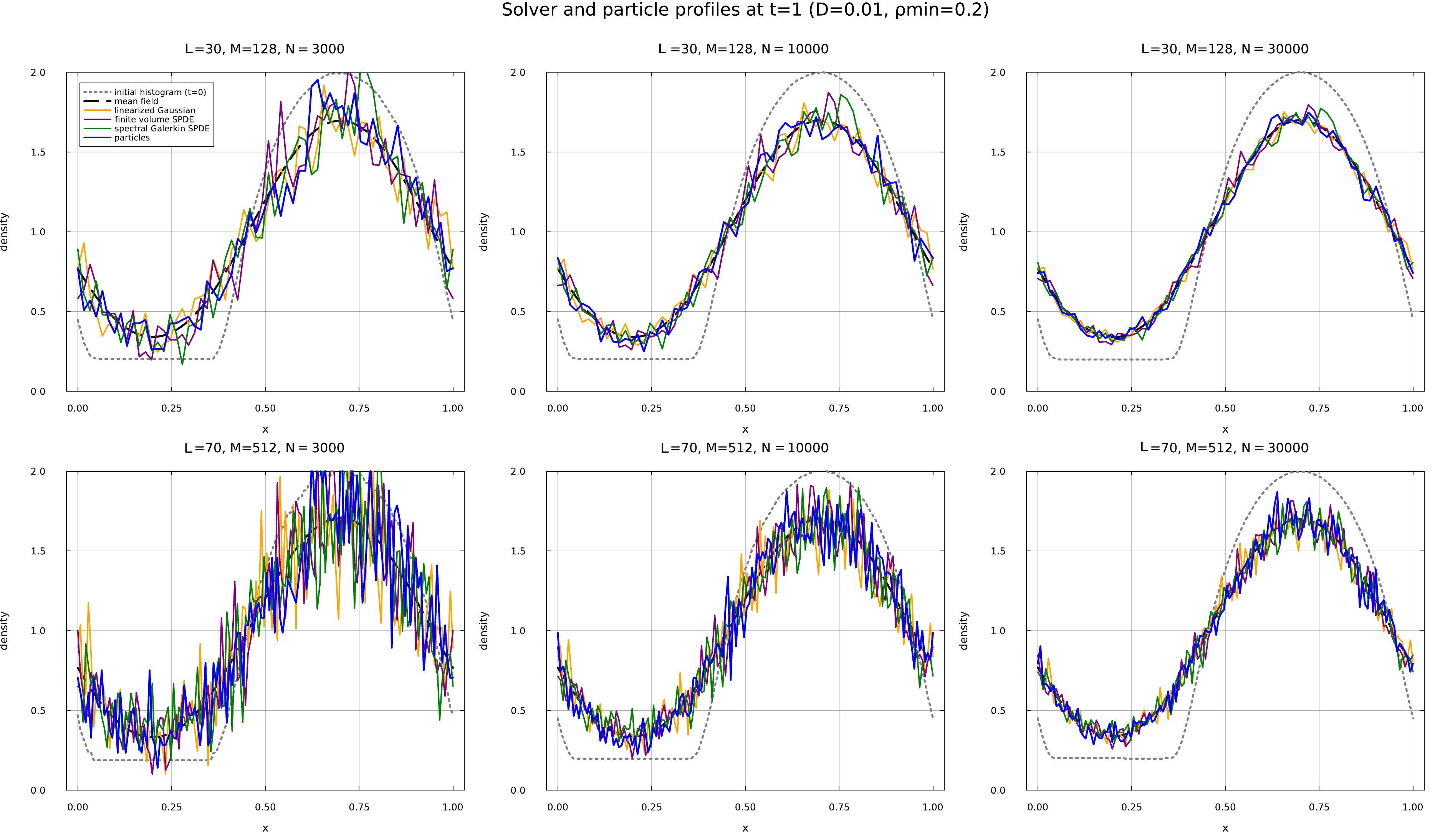}
    \caption{Sample trajectories at time $T=1$ for particle numbers
    $N=3{,}000$, $10{,}000$, and $30{,}000$ from left to right.
    The upper row uses spectral cutoff $L=30$ and nonlinear grid size
    $M=128$, while the lower row uses $L=70$ and $M=512$. Each panel
    shows the initial histogram, the deterministic finite-volume mean field,
    and single realizations of the finite-volume linearized-Gaussian,
    finite-volume Dean--Kawasaki, spectral Galerkin Dean--Kawasaki,
    and particle approximations.}
    \label{fig:sample_trajectories}
\end{figure}

This histogram is used to initialize all methods shown in Figure~\ref{fig:sample_trajectories}. For the spectral Galerkin scheme, the histogram values are treated as nodal values and transformed into Fourier coefficients. For comparison, we also show a semi-implicit finite-volume discretization with centered conservative fluxes, the corresponding deterministic finite-volume mean field, and a finite-volume linearized-Gaussian approximation. The finite-volume solvers treat the histogram as a piecewise constant initial density. The linearized-Gaussian approximation starts with zero fluctuation. We also show a particle histogram on the same cells, where particles are initially placed at the grid points according to the histogram.

Figure~\ref{fig:sample_trajectories} shows that fluctuations increase with $L$ and for moderate $N$ and large $L$ the fluctuations may dominate.

\subsection{Monte Carlo Experiment for Loss of Positivity}
\label{sec:positivity_experiment}

We next examine how frequently the spectral Galerkin approximation develops
negative values. The simulations are performed on the unit
torus $\mathbb T=\mathbb R/\mathbb Z$ with
\(
D=0.01,\,\Delta t=10^{-4},\,\text{and }T=1.
\)
We consider four spectral cutoffs $L$, with the particle-number parameter
chosen as a function of $L$.  

The nonlinear stochastic term is evaluated on a physical-space grid whose
size $M(L)$ also depends on the cutoff, with $M(L)$ being the smallest power of $2$ that is bigger than $4L+1$. The choice of the pairs $(L,N(L))$ is motivated by the expected size of
the projected fluctuations over short times, which by a variation of Corollary~\ref{cor:vunif} around a point $x$ is of the order $\rho_0(x)^{1/2} \kappa$, so that we require 
$\rho_{\min}
\gg \kappa^2 =
\frac{L\log L}{N}$, where we recall that $d=1$.

The particle numbers are chosen so that this quantity remains approximately
constant across the four cutoffs:
\[
\frac{L\log L}{N}
\approx
2.485\times10^{-3}.
\]
The experiment
is designed to investigate the changes in fluctuation behavior
produced by varying the minimum and maximum of the initial density through
$\rho_{\min}$ and $\rho_{\max}$. The parameter $\rho_{\min}$ determines the initial distance of
the density from zero. The observed dependence of the fluctuation size on $\rho_{\max}$ is also suggested by our analysis, in particular the estimate~\eqref{eq:maximum-negative-u} for the negative part. For the family of initial densities introduced above, we use $\rho_{\max}\in\{1.7,2.8,4.0\}$
and
\[
\rho_{\min}
=
s\frac{128\log(128)}{250{,}000}\approx
2.485 s \times10^{-3},
\qquad
s\in\{2,2.5,3,\ldots,15\}.
\]

The continuum density is evaluated on a uniform validation grid of
$16{,}384$ points and normalized using the periodic trapezoidal rule. A
single fine-grid Fourier transform is computed through the largest cutoff
$L_{\max}=128$. For each smaller cutoff, the initial coefficients are
obtained by truncation. Parameters for which the truncated initial density is already negative are omitted and not evolved.

The stochastic evolution is computed using the spectral Galerkin scheme described above. After each time step, the spectral density is reconstructed on the finer \(M(L)\)-point physical grid, and its minimum is evaluated. If the density is negative at any grid point, the evolution is stopped and the trajectory is recorded as negative. Otherwise, it is evolved until the terminal time \(T=1\). By repeating this experiment over \(1000\) trajectories, we estimate the probability
\[P\left( \min_{\substack{1\leq n\leq T/\Delta t\\0\leq j<M(L)}} \rho_L^n(x_j)<0 \right). \]
To select the time-step size, we performed a preliminary
comparison, not shown here, using $\Delta t=10^{-3}$, $10^{-4}$, and
$10^{-5}$ with $200$ realizations per parameter point. A noticeable change
was observed between $\Delta t=10^{-3}$ and $\Delta t=10^{-4}$, whereas the
results for $\Delta t=10^{-4}$ and $\Delta t=10^{-5}$ showed little visible
difference at this Monte Carlo resolution. We therefore use
$\Delta t=10^{-4}$ to reduce the computational cost. Figure~\ref{fig:positivity_probability} displays the dependence on $\rho_{\min}$ and $\rho_{\max}$.

\begin{figure}[!htbp]
    \centering
    \captionsetup{font=footnotesize}
    \includegraphics[
        width=\linewidth,
        height=0.82\textheight,
        keepaspectratio
    ]{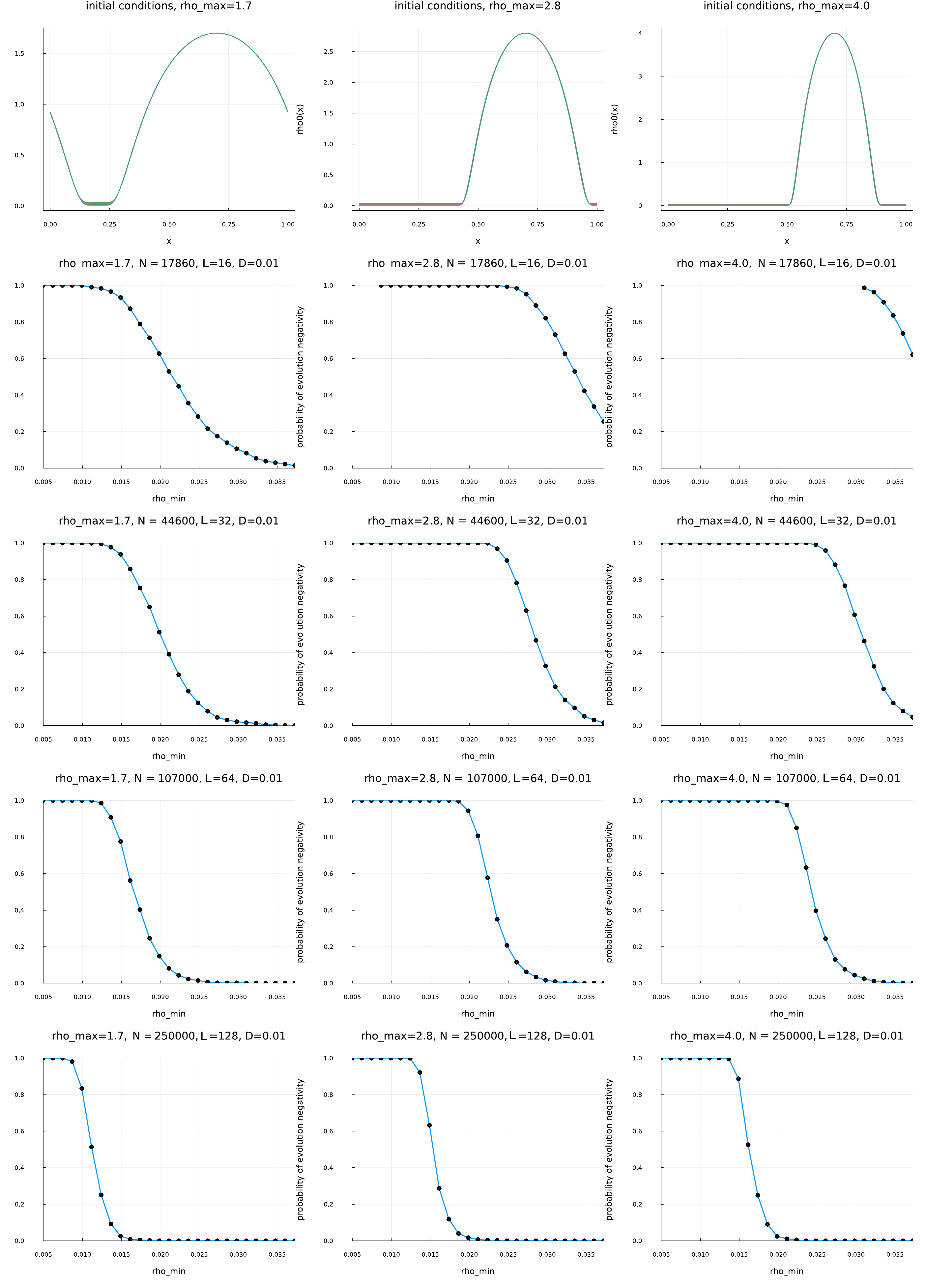}
    \caption{Dependence of the estimated probability of detected loss of
    positivity on the minimum and maximum of the initial density. 
    The first row shows the
    corresponding families of initial probability densities. The subsequent
    rows show the estimated probabilities.} 
    \label{fig:positivity_probability}
\end{figure}

\subsection{Weak Error Comparison}
\label{sec:third_moment_comparison}

We next compare the third moments obtained with the finite-volume and
spectral solvers to the exact finite-particle moments and the
linearized-Gaussian approximation. Particular care is taken in constructing
the initial data so that discrepancies already present at time $t=0$ are
negligible compared with the weak errors generated during the evolution.

 In this experiment, the initial particle histogram is constructed from the projection onto the Fourier modes
$|k|\leq10$ of the positive smooth mollifier profile defined in
\eqref{eq:positive_molifier}, with $\rho_{\min}=0.2$ and
$\rho_{\max}=2$, as described above. In the previous experiments, the same piecewise constant deterministic histogram was used to initialize the finite-volume and spectral solvers, while all particles belonging to a cell were placed at its grid point. This concentrates the whole cell mass at one point, instead of distributing it throughout the cell, which creates a noticeable discrepancy at time \(0\). To make the particle configuration more consistent with the finite-volume interpretation, we now retain the same counts per cell but distribute the particles equidistantly within their corresponding cells. 
The spectral solver is then initialized from the redistributed particle configuration via Galerkin approximation: 
\[ \widehat{\rho}_k(0) = \frac1{N} \sum_{i=1}^{N} e^{-2\pi ikX_i(0)}, \qquad |k|\leq L. \]
For the particle initial configurations that we consider, this does produce positive densities.

Thus, the finite-volume and spectral solvers use two different representations
of the same deterministic particle configuration such that they both have negligible initial error.

In the experiment we choose the strictly positive initial profile such that we are in a regime
in which loss of positivity is expected to be very unlikely. The spectral
solver is therefore tested in the regime relevant to the
$\varepsilon^\alpha$ weak-error rate established in
Corollary~\ref{cor:moments}, rather than in one dominated by initialization
effects or negative excursions.

\paragraph{Exact finite-particle reference.}

For deterministic initial positions $X_i(0)$, define 
\[
    m_{r,i}(t) = e^{tD\Delta}(\varphi^r)(X_i(0)),\qquad r=1,2,3.
\]
By particle independence, we can explicitly compute the moments of the empirical measure: 
\[
\mathbb E[\langle \mu_t, \varphi\rangle]
=\frac{1}{N}\sum_{i=1}^{N}m_{1,i}(t),\quad 
\operatorname{Var}(\langle \mu_t, \varphi\rangle)
=\frac{1}{N^2}\sum_{i=1}^{N}
\left(m_{2,i}(t)-m_{1,i}(t)^2\right),
\]
and
\[
\E[(\langle \mu_t,\varphi\rangle - \E[\langle \mu_t, \varphi\rangle])^3]
=\frac{1}{N^3}\sum_{i=1}^{N}
\left(m_{3,i}(t)-3m_{1,i}(t)m_{2,i}(t)+2m_{1,i}(t)^3\right).
\]
The semigroup is evaluated spectrally 
with Fourier coefficients computed analytically when available and by fine-grid quadrature otherwise. 

\paragraph{Exact linearized-Gaussian reference.}

The moments of the linearized-Gaussian observable can also be computed exactly
from its mean and variance. The centered third moment is identically zero by Gaussianity.

\paragraph{Test functions.}

The first four rows of Figure~\ref{fig:third_moment_weak_error} use the
exact-Hölder test functions
\[
\varphi_\alpha(x)
=
\left|\sin\bigl(\pi(x-0.7)\bigr)\right|^\alpha,
\qquad
\alpha\in\{2.5,3.7,4.9,6.1\}.
\]
The final row uses the band-limited periodic Gaussian
\[
\varphi_{\mathrm{BL}_{1}}(x)
=
1+
2\sum_{k=1}^{100}
\exp\left[-\frac12(2\pi\sigma k)^2\right]
\cos\left(2\pi k(x-0.17)\right),
\qquad
\sigma
=
\frac{0.05}{\sqrt{2\log(10^3)}}.
\]

\begin{figure}[!htbp]
    \centering
    \captionsetup{font=footnotesize}
    \includegraphics[
        width=\linewidth,
        height=0.9\textheight,
        keepaspectratio
    ]{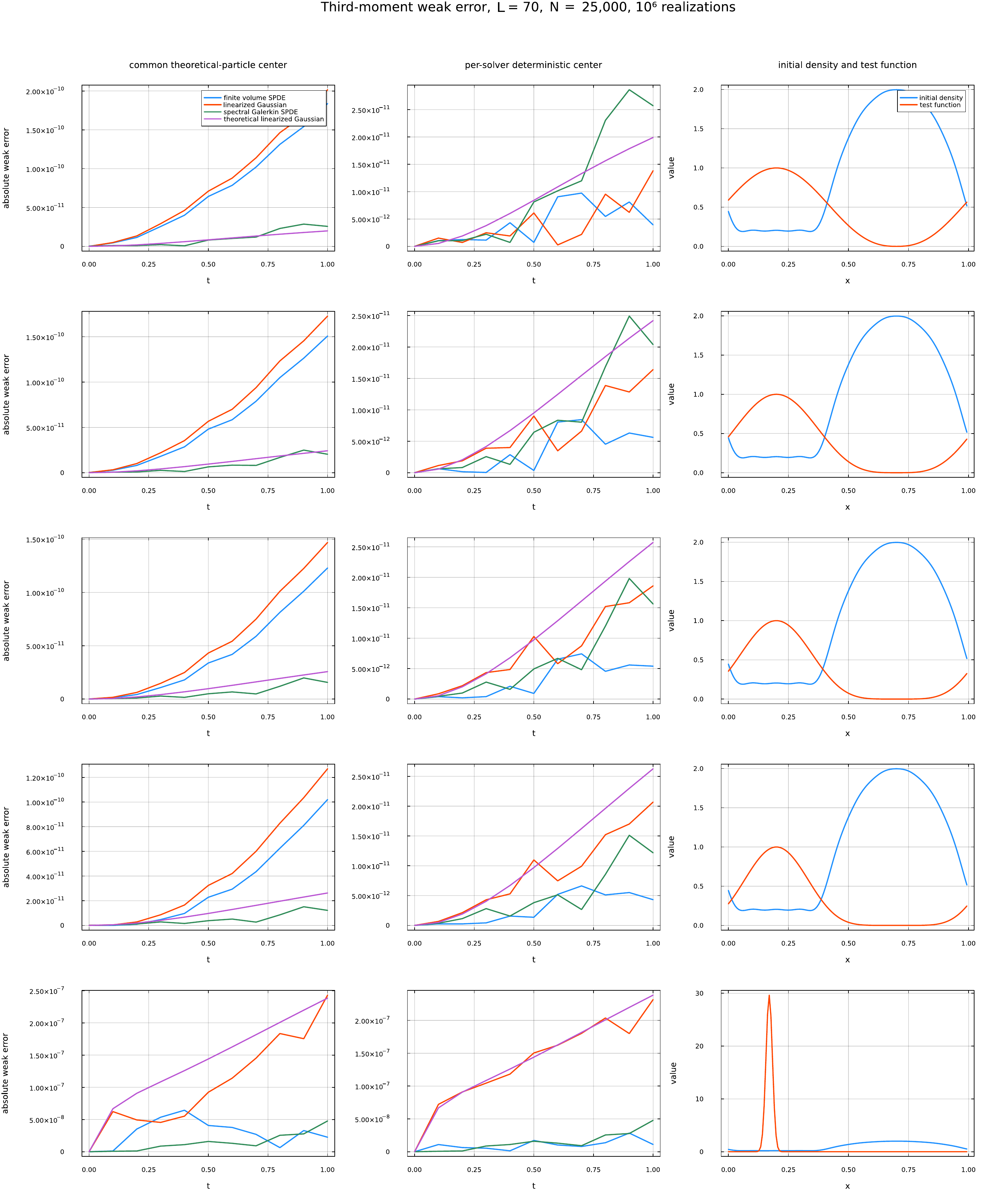}
    \caption{Third-moment weak errors for $N=25{,}000$, spectral cutoff
$L=70$, auxiliary spectral grid size $M=256$, timestep
$\Delta t=10^{-3}$, terminal time $T=1$, and $10^6$ Monte Carlo
realizations for each stochastic solver. The left column uses the exact
finite-particle mean as a common center, while the middle column uses the
deterministic mean computed by evolving the heat equation with the respective numerical solver. The right column shows, in
blue, the projection onto the Fourier modes $|k|\leq10$ of the initial
profile from~\eqref{eq:positive_molifier} with
$\rho_{\min}=0.2$ and $\rho_{\max}=2$. The projected profile has unit mass
and numerical minimum approximately $0.1893$. The corresponding test
function is shown in orange. From top to bottom, the rows use the
exact-Hölder functions with $\alpha=2.5$, $3.7$, $4.9$, and $6.1$, followed
by the band-limited Gaussian centered at $x=0.17$.} 
    \label{fig:third_moment_weak_error}
\end{figure}

Figure~\ref{fig:third_moment_weak_error} shows a clear distinction between common-center and per-solver-centered errors. With a common center, the third moment also inherits errors from the first two moments and is therefore sensitive to the numerical propagation of the deterministic heat equation. This accounts for the larger finite-volume errors compared to the spectral solver. The same mechanism explains the separation between the numerical and theoretical linearized-Gaussian curves, apart from Monte Carlo uncertainty.

Centering each solver by its own deterministic mean largely removes this effect and more directly isolates the third centered fluctuation moment. For smoother test functions, these results show that the linearized-Gaussian model does not reproduce the finite-particle third-order statistics, indicating an advantage of the nonlinear Dean--Kawasaki approximation for nonlinear observables beyond second moments.

The dependence on test-function regularity is also evident. For small $\alpha$, the slower Fourier decay limits the spectral method at fixed cutoff. As $\alpha$ increases, the spectral error decreases relative to the finite-volume and linearized-Gaussian approximations, consistently with the $\varepsilon^\alpha$ dependence in Corollary~\ref{cor:moments}.

\subsection{Convergence with Respect to the Spectral Cutoff}
\label{subsec:spectral_convergence}

We conclude by examining the weak-error convergence of the spectral
Galerkin approximation with respect to the cutoff $L$. Motivated by the example considered in \cite{Cornalba2023}, we work on the
$2\pi$-torus and use the same initial density,
\[
\rho_0(x)
=
C_\rho
\left[
3-
2\exp\left(
-\frac{\sin^6(x/2)}{0.05}
\right)
\right],
\]
where $C_\rho$ is chosen such that the density has mass $1$.
The remaining parameters are $D=0.01$, $N=11{,}000$, and $T=1$.
For each $L=2,\ldots,15$, the spectral moments are estimated using
$10^6$ Monte Carlo realizations. The finite-particle reference moments are
computed exactly, as described above.
In each panel, the dashed line is a
least-squares fit of the form
$
E_L\approx CL^{-p}.
$

\begin{figure}[!htbp]
    \centering
    \captionsetup{font=footnotesize}
    \includegraphics[width=\linewidth]{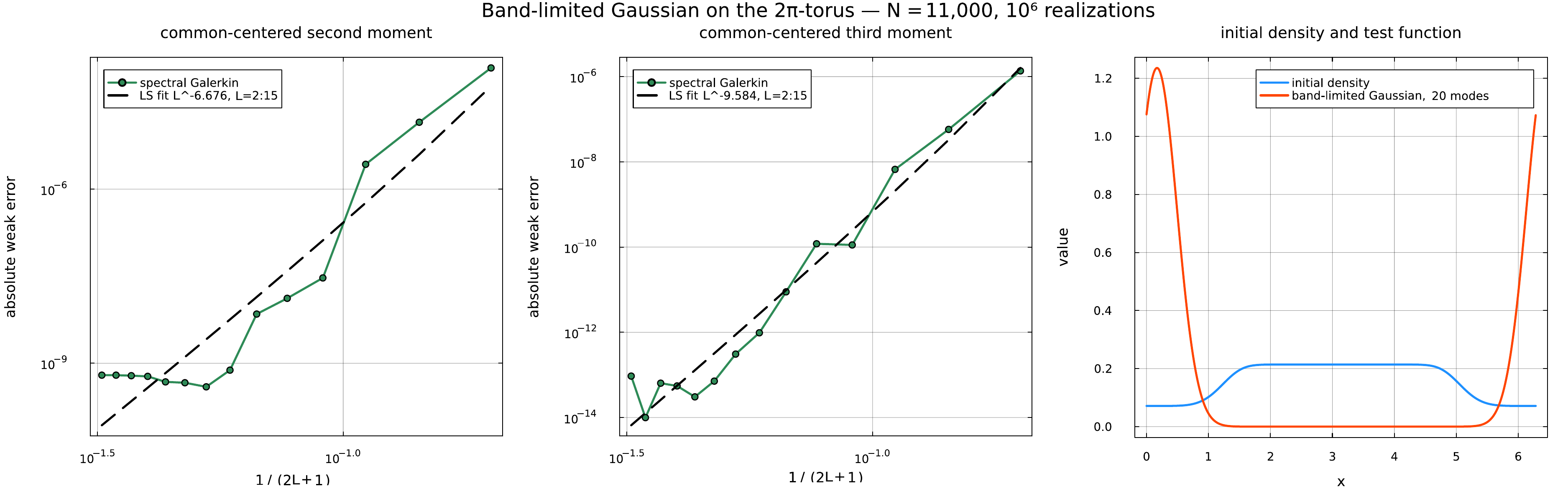}
    \caption{Common-centered spectral weak-error convergence for the
    periodic band-limited Gaussian test function. 
    The left and middle panels show the second- and
    third-moment weak errors, respectively. The green curves represent the
    spectral Galerkin errors. 
    The fitted exponents are
    $6.676$ and $9.584$, respectively. The right panel shows the normalized
    initial density and the unit-mass band-limited Gaussian centered at
    $x=0.17$ with Fourier modes $|k|\leq20$.}
    \label{fig:band_limited_convergence}
\end{figure}
Figure~\ref{fig:band_limited_convergence} considers the unit-mass
band-limited Gaussian test-function
\[
\varphi_{\mathrm{BL}_2}(x)
=
\frac{1}{2\pi}
\sum_{|k|\leq20}
\exp\left(-\frac12\sigma^2k^2\right)
e^{ik(x-0.17)},
\qquad
\sigma
=
\frac{1.2}{\sqrt{2\log(10^3)}}.
\]
Using the common theoretical-particle centering introduced above, the
empirical fits are
$
E_{2,L}^{\mathrm{common}}
\approx C_2L^{-6.676}
$ 
for the second moment and
$
E_{3,L}^{\mathrm{common}}
\approx C_3L^{-9.584}
$ 
for the third moment. Both errors decrease rapidly over the considered
range. Since the test function is smooth, the
$\varepsilon^\alpha$ estimate from Corollary~\ref{cor:moments} is available
for arbitrarily large finite values of $\alpha$. 
Nevertheless, 
arbitrarily fast rate cannot be observed
numerically.

\begin{figure}[!htbp]
    \centering
    \captionsetup{font=footnotesize}
    \includegraphics[width=\linewidth]{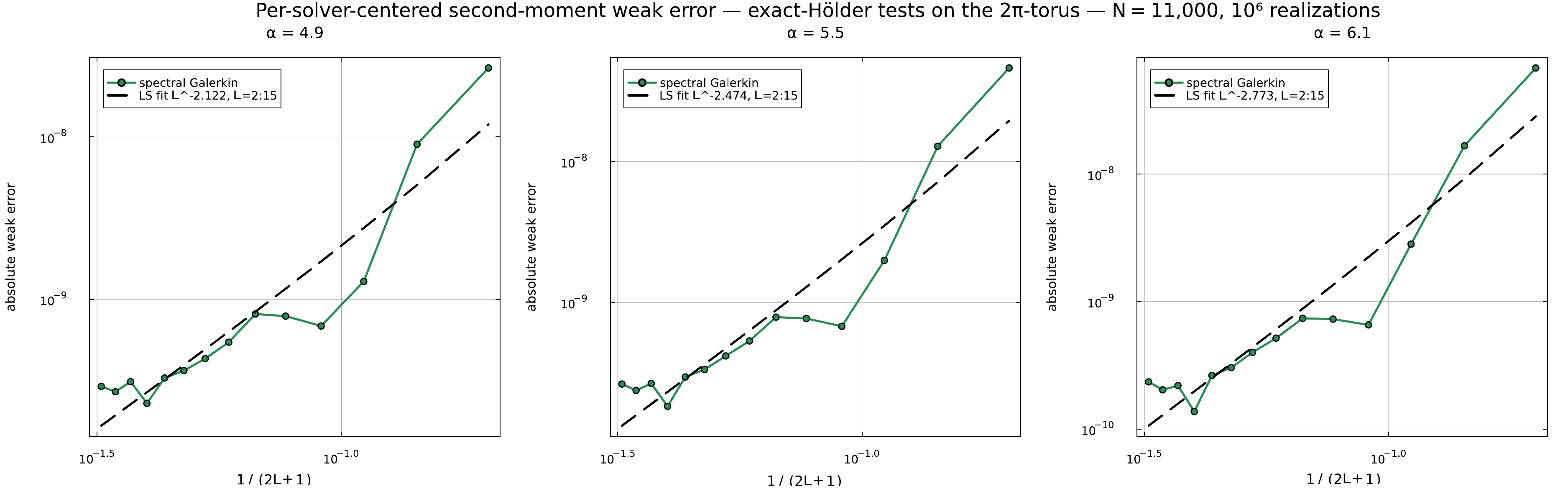}
    \caption{Per-solver-centered second-moment spectral weak errors for the
    exact-Hölder test functions
    $\varphi_\alpha(x)=|\sin((x-0.7)/2)|^\alpha$, with
    $\alpha=4.9$, $5.5$, and $6.1$ from left to right. 
    The green curves show the spectral Galerkin errors for
    $L=2,\ldots,15$, plotted against $(2L+1)^{-1}$. 
    The fitted
    exponents are $2.122$, $2.474$, and $2.773$, respectively.}
    \label{fig:holder_convergence}
\end{figure}

Figure~\ref{fig:holder_convergence} shows the per-solver-centered
second-moment error for the exact-Hölder test functions
\[
\varphi_\alpha(x)
=
\left|
\sin\left(\frac{x-0.7}{2}\right)
\right|^\alpha,
\qquad
\alpha\in\{4.9,5.5,6.1\}.
\]

We observe that the fitted exponent again increases with the smoothness of the test
function, showing the regularity dependence of the spectral weak error.
The improvement is, however, weaker than a direct
$\varepsilon^\alpha$ dependence would suggest. 
All exponents are fitted over $L=2,\ldots,15$ and may be influenced by pre-asymptotic behavior, statistical and discretization errors, cancellation, and finite precision. The results therefore indicate the qualitative dependence of the spectral weak error on test-function regularity, but do not determine a precise relation between $\alpha$ and the observed convergence rate.

\appendix
\section*{Appendix}
\section{Proof of Proposition \ref{prop:wp}}\label{Appendix_wp}

\begin{proof} 
Let $H:=L^2(\mathbb T^d)$ and $U:=L^2(\mathbb T^d; \mathbb R^d)$ and let $A:=\frac12\Delta$ with domain
$D(A)=H^2(\mathbb T^d)$.  Then $A$ generates a $C_0$-semigroup on $H$.
For $u\in H$, define $B_\varepsilon(u)\in L_2(U,H)$ (space of Hilbert–Schmidt operators from
$U$ to $H$) by $B_\varepsilon(u)g := \frac{1}{\sqrt{N}}\nabla \cdot \bigl(K_\varepsilon * \bigl(f(u^+)g)\bigr)$, for
$g\in U$. Hence, for every $u,v \in H$,
\begin{equation}\label{B_eps_diff}
    \|B_\varepsilon(u) - B_\varepsilon(v)\|_{L_2(U,H)}^2 =\frac{1}{N}\|\nabla K_\varepsilon\|_{L^2}^2 \,\|f(u^+) - f(v^+)\|_{L^2}^2,
\end{equation}
and 
\begin{equation}\label{B_eps}
    \|B_\varepsilon(u)\|_{L_2(U,H)}^2
=
\frac{1}{N}\|\nabla K_\varepsilon\|_{L^2}^2 \,\|f(u^+)\|_{L^2}^2.
\end{equation}

Assume first that  $f$ is globally Lipschitz, with Lipschitz constant $L_f$. Since $r \mapsto r^+$ is $1-$Lipschitz, 
\eqref{B_eps_diff} gives
\[
\|B_\varepsilon(u)-B_\varepsilon(v)\|_{L_2(U,H)}
\le
\frac{\|\nabla K_\varepsilon\|_{L^2}}{\sqrt{N}}\,L_f\,\|u-v\|_{L^2}.
\]
Moreover, global Lipschitz continuity implies linear growth 
and hence, by \eqref{B_eps} and using that $\T^d$ has finite measure, 
$B_\varepsilon:H\to L_2(U,H)$ is globally Lipschitz and has linear growth. The well-posedness theorem for  stochastic evolution equations with
globally Lipschitz coefficients driven by a cylindrical Wiener process from \cite[Ch. 7]{DaPrato2014} yields a unique
mild solution $u$ to \eqref{eq:mild_formulation}. It remains to show that $u$ is a weak solution. By the preceding
linear growth estimate,
\[
\mathbb{E}\int_0^T
\|B_\varepsilon(u_s)\|_{L_2(U,H)}^2\mathd s
\le
C_{\varepsilon,N,f,T}
\Bigl(1+\mathbb E\sup_{s\in[0,T]}\|u_s\|_H^2\Bigr)<\infty .
\]
Thus,  the
square-integrability assumption on the
diffusion term required in the weak--mild equivalence theorem \cite[Thm 6.5]{DaPrato2014} is satisfied and $u$ is also a weak solution in the sense of Definition~\ref{def:weak_sol}.

We now assume only that $f$ is continuous and has at most linear growth and we show that \(B_\varepsilon:H\to L_2(U,H)\) is continuous and satisfies a linear
growth bound. Using the linear growth assumption on $f$ it follows as in the Lipschitz case that $B_\varepsilon$ has linear growth.

To prove continuity, let \(u_n\to u\) in \(H\). Since \(r\mapsto r^+\)
is Lipschitz, \(u_n^+\to u^+\) in \(L^2(\mathbb T^d)\), hence in measure. By the
continuity of \(f\), we also get $f(u_n^+)\to f(u^+)$ in measure. Furthermore, the linear growth of \(f\) gives
\[
|f(u_n^+)-f(u^+)|^2
\le
C\bigl(1+|u_n|^2+|u|^2\bigr).
\]
Since \(u_n\to u\) in \(L^2\), the family \((|u_n|^2)_{n\ge1}\) is uniformly
integrable in \(L^1(\mathbb T^d)\). Hence the family $\bigl(|f(u_n^+)-f(u^+)|^2\bigr)_{n\ge1}$ is uniformly integrable in \(L^1(\mathbb T^d)\). Vitali's theorem
therefore yields $f(u_n^+)\to f(u^+)$ in $L^2(\mathbb T^d)$. Using \eqref{B_eps_diff}, we conclude that
\[
\|B_\varepsilon(u_n)-B_\varepsilon(u)\|_{L_2(U,H)}
\to 0 .
\]
Thus \(B_\varepsilon\) is continuous and has linear growth.  By the existence theorem for
stochastic evolution equations with continuous coefficients of linear growth,
\cite[Theorem 8.1]{DaPrato2014}, there exists a
probabilistically weak mild solution. As in the Lipschitz case we see that $u$ is also weak in the sense of Definition \ref{def:weak_sol}.
\end{proof}

\begin{remark}
The cited result \cite[Theorem~8.1]{DaPrato2014} for the existence of probabilistically weak mild solution is stated for deterministic initial condition and  we have used it for a random initial condition $u_0 \in L^2(\Omega; L^2)$ which is independent of $W$.  
The proof extends to this setting by a similar tightness argument as in~\cite[Chapter~8]{DaPrato2014}, where we first approximate $u_0$ by a random variable in $L^p(\Omega;L^2)$. 
\end{remark}

\section{Auxiliary estimates}\label{AppendixA}

Here we collect some technical estimates that we use in the main proofs.

\begin{lemma}[Heat-kernel derivative bound]\label{lem:q-estimate} Let $K_\varepsilon = \mathcal F^{-1} (\mathds{1}_{[-1,1]}(\varepsilon\cdot))$ on the one-dimensional torus and let $q^\varepsilon(r,\cdot) = \partial_x K_\varepsilon\ast p_r$.  There exists $C>0$ such that for all sufficiently small $\varepsilon>0$, for all $\tau>0$, and for all
$
r\in (0, \tau(\tfrac{\varepsilon}{2\pi})^2]
$ 
and all
$ 
z\in\left[\frac{\varepsilon}{2\pi},\frac{\varepsilon}{\pi}\right],
$
it holds
\[
q^\varepsilon(r,z)
\leq
-C e^{-\tau}\varepsilon^{-2}.
\]
\end{lemma}

\begin{proof}
 By definition,
\[
q^\varepsilon(r,z) 
= 
\sum_{|k|\leq \varepsilon^{-1}}
2\pi i k\, e^{-4\pi^2 k^2 r}e^{2\pi i kz} =-4\pi
\sum_{1\leq k\leq \varepsilon^{-1}}
k e^{-4\pi^2 k^2 r}\sin(2\pi k z).
\]  
Let
$ 
z\in \left[\frac{\varepsilon}{2\pi},\frac{\varepsilon}{\pi}\right],
$ 
then for every integer $k \in \left[
\frac{1}{2\varepsilon}, \frac{1}{\varepsilon}\right]$,
we have $2\pi k z\in [1/2,2]$. Since $[1/2,2]\subset(0,\pi)$,  we obtain $\sin(2\pi k z)\geq \sin(1/2)>0$
for all such $k$ and $z$. Furthermore, notice that $\sin(2\pi k z)\geq 0 $ for $z \in [\tfrac{\varepsilon}{2\pi},\tfrac{\varepsilon}{\pi}]$ and $k \leq \frac{1}{\varepsilon}$. Therefore, we can restrict the Fourier sum to the modes $k \in [
\tfrac{1}{2\varepsilon}, \tfrac{1}{\varepsilon}]$, which gives
\[
q^\varepsilon(r,z)
\leq
-C
\sum_{(2\varepsilon)^{-1}\leq k\leq \varepsilon^{-1}}
k e^{-4\pi^2 k^2r}.
\]
Assume now that
$r\in (0, \tau (\tfrac{\varepsilon}{2\pi})^2)$.
For every $k\leq\varepsilon^{-1}$, we have $e^{-4\pi^2 k^2r}
\geq e^{-\tau}$. It follows that
\[
q^\varepsilon(r,z)
\leq
-Ce^{-\tau}
\sum_{(2\varepsilon)^{-1}\leq k\leq \varepsilon^{-1}} k \leq -Ce^{-\tau}\varepsilon^{-2},
\]
for all sufficiently small \(\varepsilon>0\).
\end{proof}

\begin{lemma}\label{lem:K-eps-error}
    For $\alpha>0$, $p\in[1,\infty]$ and $\varphi \in B^\alpha_{p,\infty}$, we have
    \[
        \|\varphi - K_\varepsilon\ast \varphi\|_{L^p} \lesssim \varepsilon^\alpha \|\varphi\|_{B^\alpha_{p,\infty}}.
    \]
\end{lemma}

\begin{proof} We use that $\mathcal{F} K_{\varepsilon}$ is equal to $1$ on a
  ball of order $\varepsilon^{- 1}$, and therefore with the Littlewood-Paley blocks $(\Delta_j)_{j \geq-1}$ and for $2^{j_0} \simeq
  \varepsilon^{-1}$ we obtain with Bernstein's inequality, see \cite[Lemma 2.1]{Bahouri2011} for a version on $\R^d$ and \cite[Lemma~7]{Gubinelli2015EBP} for the version on $\T^d$ that we apply here,
  \[ \| \varphi - K_{\varepsilon} \ast \varphi
     \|_{L^p} \leq \sum_{j \geq j_0} \| (1 -
     K_{\varepsilon}\ast) \Delta_j \varphi \|_{L^p} \lesssim
     \sum_{j \geq j_0} 2^{-j \alpha} \| \varphi \|_{B^{\alpha}_{p,\infty}} \lesssim \varepsilon^{\alpha} \| \varphi
     \|_{B^{\alpha}_{p,\infty}}. \qedhere
  \]
\end{proof}

\begin{lemma}
  \label{lem:tail-estimate}
  There exist universal constants $c,C>0$, independent of the quantities introduced below, such that all of the following statements hold: If $\beta, \gamma > 0$ and $X$ is a real-valued random variable such that, for every $p\ge 2$,
  \[ \mathbb{E} [| X |^p]^{1 / p} \leq \beta \sqrt{p} + \gamma p,
  \]
  then
  \[ \mathbb{P} (| X | \geq x) \leq C \exp \left( - c
     \left( \frac{x^2}{\beta^2} \wedge \frac{x}{\gamma} \right) \right),
  \]
  for all $x \geq 0$, and moreover,  for every $\lambda \in [0, c / \gamma]$, 
  \[ \mathbb{E} [\exp (\lambda | X |)] \leq C \exp ( C \lambda^2 \beta^2) . \]
  Conversely, if for $K,\beta,\gamma>0$
  \[
    \mathbb P(|X|\geq x) \leq K \exp \left( - \left( \frac{x^2}{\beta^2} \wedge \frac{x}{\gamma} \right) \right),\qquad x \geq 0,
  \]
  then it holds for all $p\geq 1$
  \[
    \mathbb{E} [| X |^p]^{1 / p} \leq CK^{1/p} (\beta \sqrt{p} + \gamma p).
  \]
\end{lemma}

\begin{proof}
  We first prove the tail estimates under the assumption $\mathbb{E} [| X |^p]^{1 / p} \leq \beta \sqrt{p} + \gamma p$. If
  \[
        p_0 := \frac{1}{4} \left( \frac{x^2}{\beta^2} \wedge \frac{x}{\gamma} \right) < 2,
  \]
  we can enforce the bound by choosing $C>0$ large enough. So let $p_0 \ge 2$ and apply Markov's inequality together with the assumed moment bound for $X$ with $p=p_0$:
  \[ \mathbb{P} (| X | \geq x) \leq \left( \frac{
         \beta \sqrt{p_0}}{ x}  + \frac{\gamma p_0}{x}\right)^{p_0} \leq \left(\frac12 + \frac{1}{4}\right)^{p_0} = \left(\frac34\right)^{p_0} = \exp\left(\frac14 \log \left(\frac34 \right)  \left( \frac{x^2}{\beta^2} \wedge \frac{x}{\gamma} \right)\right) ,
  \]
  so that we can take $c=-\frac14 \log \frac34>0$. 
  
  We now prove the exponential-moment estimate. For every integer $k \geq 2$, the moment assumption gives
  \[
  \mathbb{E}|X|^k \leq (\beta\sqrt{k} + \gamma k)^k \leq 2^{k}( \beta^k k^{k/2}+\gamma^k k^k).
  \]
  Hence, using 
   Stirling's inequality $k! \geq (k / e)^k$ and for $\lambda\ge 0$ such that $2e\gamma\lambda\leq 1/2$ so we can bound the geometric series appearing below by $1$ and so that also $\lambda\gamma\le 1$: 
\begin{align*}
       \mathbb{E} [\exp (\lambda | X |)] & = 1+ \lambda\mathbb E[|X|] + \sum_{k \geq 2}
    \frac{\lambda^k}{k!} \mathbb{E} [| X |^k] \nonumber \\
    & \lesssim 
    1+ \lambda(\beta+\gamma) + \sum_{k \geq 2} \frac{\lambda^k}{k!} 2^k (\beta^k k^{k / 2} + \gamma^k k^k) \nonumber \\
    & \leq 
    1+ \lambda(\beta+\gamma) + \sum_{k \geq 2} \frac{(2 e
    \lambda \beta)^k} { k^{k / 2}} + \sum_{k \geq 2} (2e\gamma \lambda)^k \\ 
    & \lesssim 1+\sum_{k \geq 0} \frac{(2 e
    \lambda \beta)^k} { \sqrt{k!}}  \lesssim  \left(\sum_{k \geq 0} \frac{(4 e
    \lambda \beta)^{2k}} { k!}\right)^{1/2} \left(\sum_{k \geq 0} 2^{-2k}\right)^{1/2} \lesssim \exp\left(\frac12 (4e\lambda\beta)^2\right),
\end{align*} 
which is the claimed bound for the exponential moment.

It remains to show the converse direction:
\begin{align*}
    \mathbb E[|X|^p] & = \int_0^\infty p x^{p-1} \mathbb P(|X| \ge x) \mathd x \\
    & \leq \int_0^\infty p x^{p-1} K \exp\left( - \frac{x^2}{\beta^2} \right) \mathd x + \int_0^\infty p x^{p-1} K \exp\left( - \frac{x}{\gamma} \right) \mathd x \\
    & = K p \beta^{p} \int_0^\infty x^{p-1} \exp(- x^2)\mathd x + Kp\gamma^p \int_0^\infty x^{p-1} \exp(-x)\mathd x \\
    & \lesssim K p \beta^{p} C^p p^{p/2} + K p\gamma^p C^p p^p,
\end{align*}
for some $C>0$, by the scaling in $p$ of the moments of normal and exponential random variables.
\end{proof}

\section*{AI declaration}
In preparing this work we used ChatGPT, versions 5.2--5.5, via ChatGPT and Codex, for the following purposes: discussion of parts of the proof strategy, in the course of which the model suggested an approach for some arguments that we then worked out and wrote up in full; clarifying the exposition of proofs; improving phrasing; producing a first draft of parts of Section~\ref{sec:numerics}, which we revised; and assisting with the numerical experiments. We also used Claude, version Opus 4.8 and 5, for proofreading. All proofs were written by the authors, who take full responsibility for the contents of this paper.

\section*{Acknowledgment}
The authors gratefully acknowledge funding   by Deutsche Forschungsgemeinschaft (DFG)  through  CRC 1114 ``Scaling Cascades in Complex Systems", Project Number 235221301, Project C10 ``Numerical
Analysis for nonlinear SPDE models of particle systems" and through  IRTG 2544 Stochastic Analysis in Interaction (project ID 410208580). AD is additionally grateful for the support from the Berlin Mathematical School (BMS), which is
funded by DFG under Germany's Excellence Strategy – The Berlin Mathematics 
Research Center MATH+ (EXC-2046/1, EXC-2046/2, project ID: 390685689). This material is based upon work supported by the National Science Foundation under Grant No. DMS-2424139, while ADj and NP were in residence at the Simons Laufer Mathematical Sciences Institute in Berkeley, California, during the Fall 2025 semester.

\bibliographystyle{amsalpha}   
\bibliography{DamnjanovicDjurdjevacPerkowski}

\end{document}